\documentclass[12pt]{amsart}
\usepackage{amsmath,amssymb,amsthm,color, bm}
\usepackage[colorlinks=true,urlcolor=airforceblue,citecolor=airforceblue,linkcolor=airforceblue,linktocpage,pdfpagelabels,bookmarksnumbered,bookmarksopen]{hyperref}
\definecolor{airforceblue}{rgb}{0.36, 0.54, 0.66}
\usepackage[english]{babel}
\usepackage[left=2.9cm,right=2.9cm,top=2.8cm,bottom=2.8cm]{geometry}

\numberwithin{equation}{section}

\newtheorem{theorem}{Theorem}[section]
\theoremstyle{plain}
\newtheorem{lemma}[theorem]{Lemma}
\theoremstyle{plain}
\newtheorem{proposition}[theorem]{Proposition}
\theoremstyle{plain}
\newtheorem{corollary}[theorem]{Corollary}
\newtheorem{definition}[theorem]{Definition}
\theoremstyle{definition}
\newtheorem{remark}[theorem]{Remark}
\newtheorem{example}[theorem]{Example}

\DeclareMathOperator{\Div}{div}
\DeclareMathOperator{\curl}{curl}

\newcommand{\e}{{\mathsf e}}
\newcommand{\N}{{\mathbb N}}
\newcommand{\Z}{{\mathbb Z}}
\newcommand{\R}{{\mathbb R}}

\newcommand{\eps}{\varepsilon}
\newcommand{\beq}{\begin{equation}}
\newcommand{\eeq}{\end{equation}}
\renewcommand{\le}{\leqslant}
\renewcommand{\ge}{\geqslant}

\newcommand{\cal}{\mathcal}

\def\XXint#1#2#3{{\setbox0=\hbox{$#1{#2#3}{\int}$ }
\vcenter{\hbox{$#2#3$ }}\kern-.6\wd0}}

\makeatletter
\newcommand{\opnorm}{\@ifstar\@opnorms\@opnorm}
\newcommand{\@opnorms}[1]{%
  \left|\mkern-1.5mu\left|\mkern-1.5mu\left|
   #1
  \right|\mkern-1.5mu\right|\mkern-1.5mu\right|
}
\newcommand{\@opnorm}[2][]{%
  \mathopen{#1|\mkern-1.5mu#1|\mkern-1.5mu#1|}
  #2
  \mathclose{#1|\mkern-1.5mu#1|\mkern-1.5mu#1|}
}
\makeatother

\newenvironment{enumroman}{\begin{enumerate}

}{\end{enumerate}}

\title[Uniform ellipticity and regularity of the stress field]{A general notion of uniform ellipticity \\ and the regularity of the stress field \\ for elliptic equations in divergence form}

\author[U. Guarnotta]{Umberto Guarnotta}
\address[U. Guarnotta]{Dipartimento di Matematica e Informatica, Universit\`a degli Studi di Catania, Viale A. Doria 6, 95125 Catania, Italy}
\email{umberto.guarnotta@phd.unict.it}

\author[S. Mosconi]{Sunra Mosconi}
\address[S. Mosconi]{Dipartimento di Matematica e Informatica, Universit\`a degli Studi di Catania, Viale A. Doria 6, 95125 Catania, Italy}
\email{sunra.mosconi@unict.it}

\subjclass[2010]{}

\begin{document}

\begin{abstract}
For solutions of $\Div(DF(Du))=f$  we show that the quasiconformality of $z\mapsto DF(z)$ is the key property  leading to the Sobolev regularity of the stress field $DF(Du)$, in relation with the summability of $f$. This class of nonlinearities encodes in a general way the notion of uniform ellipticity and encompasses all  known instances where the stress field is known to be Sobolev regular.  We provide examples showing the optimality of this assumption and present  two applications:  a nonlinear Cordes condition for equations in divergence form and some partial results on the $C^{p'}$-conjecture.



\end{abstract}

\maketitle

	\begin{center}
		\begin{minipage}{9cm}
			\small
			\tableofcontents
		\end{minipage}
	\end{center}
	
\section{Introduction}

In this work we are interested in Sobolev regularity results for the often called ``stress field'' $DF(Du)$ corresponding to solutions of  
\beq
\label{diveq}
\Div(DF(Du))=f,
\eeq
seen as the Euler-Lagrange equation for the energy functional 
\beq
\label{defJ}
J(w, \Omega)=\int_{\Omega}F(Dw)+f\, w\, dx,
\eeq
where $\Omega\subset \R^{N}$, $N \ge 2$, $f\in L^{m}(\Omega)$ for some $m>1$, and $F\in C^1(\R^N)$ is a strictly convex function obeying a suitable local form of uniform ellipticity condition. Questions regarding regularity of the stress field recently gained increasing interest as a basic tool to attack further regularity and locality properties of solutions to divergence form equation, see e.\,g.\,\cite{AKM, BDW, BCDKS, BCDS, CFR, CF, KM}. Despite its usefulness, however, most of the results are constrained to special kinds of nonlinearities of either $p$-Laplacian type or having Uhlenbeck structure. Simple nonlinearities such as, for example,
\beq
\label{ex0}
F(z)=|z-z_{0}|^{p}+|z+z_{0}|^{p}, \qquad 1<p<2, \quad z_{0}\ne 0,
\eeq
do not fall within most of the currently available regularity theory. Since we focus on the stress field instead of  $Du$ itself, we briefly justify this point of view recalling the general situation for functionals of the Calculus of Variations and their local minimizers. 

Let $u$ solve \eqref{diveq}. If $F\in C^{2}(\R^{N})$ is uniformly elliptic, i.\,e.
\[
\lambda\, |\xi|^{2}\le (D^{2}F(z)\, \xi, \xi)\le \Lambda\, |\xi|^{2},\qquad \forall z, \xi\in \R^{N},
\]
it is a classical result that $f\in L^{2}_{{\rm loc}}(\Omega)\Leftrightarrow u\in W^{2, 2}_{{\rm loc}}(\Omega)$, and in this case the regularity of $Du$ and $V=DF(Du)$ coincide, since $Du=DF^{-1}(V)$ and $DF$ is  bi-Lipschitz. Regarding $W^{2, m}$ regularity for $m\ne 2$,  it is clear that if $u\in W^{2, m}_{{\rm loc}}(\Omega)$, then $DF(Du)\in W^{1, m}_{{\rm loc}}(\Omega)$, and thus $f\in L^{m}_{{\rm loc}}(\Omega)$. Conversely, suppose $f\in L^m_{\rm loc}(\Omega)$: differentiating \eqref{diveq} gives 
\[
\Div \big(D^{2}F(Du)\, D v_{k}\big)=\partial_{k}f ,\qquad v_{k}=\partial_{k}u, \quad k=1, \dots, N.
\]
If $f\in L^{m}_{{\rm loc}}(\Omega)$ for $m>N$, then $Du$ is H\"older continuous by De Giorgi-Nash's theorem; so, freezing the (now continuous) coefficients of the matrix $D^{2}F(Du)$, we can apply the Calder\'on-Zygmund theory to obtain that $u\in W^{2, m}_{{\rm loc}}(\Omega)$, i.\,e.\,$DF(Du)\in W^{1, m}_{{\rm loc}}(\Omega)$.

Next, consider the $p$-Poisson equation
\beq
\label{deltap}
\Delta_{p}u=f,\qquad p>1,
\eeq
corresponding to the integrand $F(z)=|z|^{p}/p$. In this case $F$ satisfies
\[
 \lambda\, |z|^{p-2}\, |\xi|^{2}\le (D^{2}F(z)\, \xi, \xi)\le \Lambda\, |z|^{p-2}\, |\xi|^{2},
\]
and the Sobolev regularity of $Du$ is much more involved. We are aware of only one result giving second order Sobolev regularity of $u$ from $L^{m}$ assumption on $f$: in the non-degenerate case $p\in (1, 2]$, the regularity $u\in W^{2, p}(\R^{N})$ if $f\in L^{p'}(\R^{N})$ ($\frac{1}{p}+\frac{1}{p'}=1$, as usual) has been proved in \cite{Si} for global solutions and in \cite{DT} for local ones.

Through difference quotients and Caccioppoli inequalities one can usually infer Sobolev regularity of $Du$ from Sobolev regularity of $f$. In this framework, \cite{C, MRS} treat the case when $p>2$ is near uniform ellipticity proving, respectively, $u\in W^{2, 2}_{{\rm loc}}(\Omega)$ for $2\le p<3$ and $u\in W^{2, m}_{{\rm loc}}(\Omega)$ for any $m$, as long as $p-2$ is sufficiently small. The postulated regularity for $f$ is $f\in W^{1, 2}_{{\rm loc}}(\Omega)$ in \cite{C} and $f\in W^{1, m}_{{\rm loc}}(\Omega)$ for $m>N$ in \cite{MRS}.

For $p> 2$ and only assuming $L^{m}$ regularity of $f$ in \eqref{deltap}, the best results available prove {\em fractional} differentiability of $Du$, see \cite{M, M2, Mis, Sa, Si}. The main idea of \cite{KM, M, M2}, which goes back to Uhlenbeck \cite{U}, is to study the regularity properties of the field 
\beq
\label{defvp}
{\cal V}=|Du|^{\frac{p-2}{2}}\, Du
\eeq
and deduce from the latter suitable regularity of $Du$. This approach is nowadays widespread, but still failed to produce estimates in terms of Lebesgue norm of $f$ paralleling the second-order Calder\'on-Zygmund theory depicted above in the non-degenerate case.

An alternative route is  to consider the regularity of the stress field 
\[
V=|Du|^{p-2}\, Du,
\]
which arises as an interesting object {\em per se} in a variety of situations, e.\,g.\,in the framework of nonconvex variational problems \cite{CMu} and in the dual formulation of traffic congestion problems (see \cite{BCS}). In particular,  the applicability of the DiPerna-Lions theory in the latters is tied to the Sobolev regularity of the stress field of the dual problem, which has been the main concern of \cite{BCS} for a very degenerate functional of the form \eqref{defJ}.
When $f$ is Sobolev regular, variants of Caccioppoli inequalities and difference quotients methods have been used in the cited works to obtain Sobolev regularity of $V$ (see also  \cite{DS}).
 
 For less regular $f$ the seminal paper is  \cite{L}, treating the case $f\in L^m$ with $m\ge \max\{2, N/p\}$. In more recent years, the regularity of the stress field  has been the object of fruitful investigations,  also thanks to the fact that it seems to provide more natural estimates than \eqref{defvp}. Starting from \cite{DKS},  a first-order nonlinear Calder\'on-Zygmund theory for the $p$-Laplacian problem with right-hand side in divergence form
\[
\Div |Du|^{p-2}\, Du=\Div G
\]
is nowadays well developed, showing the principle  that the divergence operator can be ``canceled out'' to get estimates for $|Du|^{p-2}\, Du$ in terms of $G$ in the same space. We refer to \cite{BDW, BCDKS, BCDS} and the literature therein for this line of research, but let us remark that the order of differentiability for $V$ considered in these works is always less than $1$. 

Indeed, regarding the second-order Calder\'on-Zygmund theory (i.\,e.\,full Sobolev regularity for $V$), much less is known. A natural conjecture for solutions of \eqref{deltap}  via the same principle would be  
\beq
\label{conj}
f\in L^{m}_{{\rm loc}}(\Omega)\quad \Leftrightarrow\quad V\in W^{1, m}_{{\rm loc}}(\Omega).
\eeq
The case $m=+\infty$, $p>2$, corresponds to the well-known $C^{p'}$ conjecture, which will be discussed later, proved in the plane in \cite{ATU}. The endpoint $m=1$ of \eqref{conj}  is considered by \cite{AKM} where, e.\,g., for $f\in L^{1}_{{\rm loc}}(\Omega)$, it is proved that $V\in W^{1-\eps, 1}$ for all $\eps>0$ (actually, $f$ can be a general Radon measure in \cite{AKM}). 

The only positive result of the type \eqref{conj} (beyond the close one \cite{L}) concerns the Hilbertian case $m=2$, which has been recently proved in \cite{CM}  for equations having {\em Uhlenbeck structure}, i.\,e., of the form
\beq
\label{cm}
\Div (a(|Du|)\, Du)=f \quad (\text{or}\ \Div G).
\eeq
The role of this structural assumption is  in fact the main motivation of this work: indeed, (with the exception of \cite{M, AKM}), the higher order Calder\'on-Zygmund theory exposed so far is restricted to equations of the form \eqref{cm} and in this case it can be actually extended to systems (see \cite{BCDKS, BCDS, CM4} and the recently appeared \cite{BCDM}). As the $L^{2}$-theory seems to be the basic step to deal with the general problem \eqref{conj}, it is worth investigating {\em to what extent the Uhlenbeck structure is necessary to develop such a theory} and whether the general nonlinear problem \eqref{diveq} enjoys the same Sobolev regularity for its stress field $V=DF(Du)$. 

It turns out that such regularity holds true when the map $z\mapsto DF(z)$ is {\em quasiconformal}. A quasiconformal map  $G:\R^N\to \R^N$ is a homeomorphism belonging to $W^{1, N}_{\rm loc}(\R^N)$  such that, for some finite $K$,
\beq
\label{qcd}
|\lambda_{\rm max}(DG(z))|^N\le K\,  |{\rm det}\, DG(z)|
\eeq
almost everywhere, with $\lambda_{\rm max}(DG)$ being the maximum singular value of $DG$; we refer to  \cite{Ma} for a nice survey on quasiconformal mappings. The main outcome of our results is that quasiconformality of $DF$ in \eqref{diveq} is a natural and robust notion of uniform ellipticity, flexible enough to encompass anisotropic examples such as  \eqref{ex0}, still allowing a reasonable regularity theory.

Quasiconformal maps which are gradients of convex functions (or being, more generally, monotone) have been systematically studied by Kovalev and Maldonado in \cite{KoMa, K}; convex potentials of quasiconformal mappings are called {\em quasiuniformly convex} functions.

\begin{definition}\label{defKell}
A convex $F\in C^1(\R^N)$ is called {\em $K$-quasiuniformly convex}  for some $K\in [1, +\infty)$ if it is not affine, $DF\in W^{1, 1}_{\rm loc}(\R^N)$, and 
\beq
\label{Kell}
\lambda_{{\rm max}}(D^2F(z))\le K\, \lambda_{{\rm min}}(D^2F(z))\qquad \text{for a.\,e.\,$z\in \R^N$},
\eeq
where $\lambda_{\rm min}(M)$ and $\lambda_{\rm max}(M)$ denote the minimum and maximum eigenvalues of $M$.
\end{definition} 

For the most part, this will be the main assumption on $F$ for the study of \eqref{diveq}, and we will use the acronym `q.\,u.\,c.'\, for `quasiuniformly convex'. Clearly, \eqref{Kell} and \eqref{qcd} for $G=DF$ are equivalent (but \eqref{Kell} implies \eqref{qcd} with a constant $K^{N-1}$); in Proposition \ref{proqu} we will show that any q.\,u.\,c.\,function is strictly convex and of $(p, q)$-growth. In section \ref{examples} we will discuss concrete examples but, in the meantime, we note that $ z\mapsto |z|^p$ is q.\,u.\,c.\,for any $p>1$, and sum of q.\,u.\,c.\,functions is q.\,u.\,c., hence \eqref{ex0} is q.\,u.\,c..

\subsection{Main results}
We now present our main results, referring to the appropriate theorems in the following sections for more complete statements. The first one concerns local minimizers $u\in W^{1,p}_{\rm loc}(\Omega)$ for the functional $J$ in \eqref{defJ}, i.\,e.\,those $u$ obeying $J(u)<+\infty$ and $J(u,B)\le J(u+w,B)$ for all $B \Subset \Omega$ and $w\in W^{1, p}_0(B)$.

\begin{theorem}[Theorem \ref{minloc}]\label{minlocint}
Let $\Omega\subseteq \R^{N}$ be open, $F\in C^{1}(\R^{N})$ be  $K$-q.\,u.\,c.\,, and $f\in L^{2}(\Omega)\cap W^{-1, p'}(\Omega)$ for $p=1+1/K$. Then any local minimizer $u\in W^{1, p}_{\rm loc}(\Omega)$ of $J$  fulfills
\[
\|DF(Du)\|_{W^{1,2}(B_{R})}\le C\Big(1+\|f\|_{L^{2}(B_{2R})}+\|F(Du)\|_{L^{1}(B_{2R})}^{K}\Big)
\]
for some $C=C(K, N, R)$ and all $B_{4R}\subseteq \Omega$.
\end{theorem}

\begin{remark}\label{rem0}
Let us briefly discuss the main assumptions. Other comments may be found in Remark \ref{rempq}.
\begin{itemize}
\item
{\em Assumptions on $F$}. 
 The Sobolev regularity assumption $DF\in W^{1, 1}_{\rm loc}(\R^N)$ is necessary, as shown in example \ref{excantor}, where a strictly convex, radial $F\in C^1(\R^N)$ obeying \eqref{Kell} is constructed in such a way that the stress field of a suitable solution to \eqref{diveq} with $f=0$ is not absolutely continuous. 

The q.\,u.\,convexity condition fails in simple examples such as the orthotropic $p$-Laplacian related to the integrand
\[
F(z)=\sum_{i=1}^{N}|z_{i}|^{p}.
\]
We remark that  Giaquinta's example \cite{Gi} in $\R^6$ is a local minimizer of an analytic functional whose integrand is not q.\,u.\,c.\,and such that the stress field is in $W^{1, s}_{\rm loc}(\R^N)$ only for $s<5/4$. Playing around with examples of similar structure suggests intricate interplays between the maximal Sobolev regularity of the stress field and the possibly non-standard structure of $DF$; hence, it is not clear what to expect from functionals with non-quasiconformal gradient mapping.

In example \ref{exuhl} we discuss integrands of the form $F(z)=F(|z|)$, while examples \ref{Cir}, \ref{exani} investigate more general anisotropic functionals.

\item
{\em Assumptions on $f$}.
The hypothesis $f\in W^{-1, p'}(\Omega)$ has been made for  expository reasons and the relation between $p$ and $K$ will be derived in Proposition  \ref{proqu}. On one hand, our results are mostly local in nature and therefore it suffices to require $f\in L^2_{\rm loc}(\Omega)\cap W^{-1, p'}_{\rm loc}(\Omega)$, meaning with the latter  the intersection of the spaces $W^{-1, p'}(\Omega_n)$ on an exhausting sequence of open relatively compact $\Omega_n\uparrow\Omega$. For $p\ge 2\, N/(N+2)$, the condition  $f\in W^{-1, p'}_{\rm loc}(\Omega)$ automatically follows from Sobolev Embedding and $f\in L^2_{\rm loc}(\Omega)$; for $p<2\, N/(N+2)$, $L^{2}(\Omega)$ is not embedded in $W^{-1, p'}(\Omega)$, destroying the variational framework we chose to be in. One should then resort to the notion of {\em approximable solutions} (briefly described in the last section), in order to deal with those cases. We refer to  \cite{ACC} for a comprehensive theory of approximable solutions  in the anisotropic framework. Anyway, in terms of summability, $f\in W^{-1, p'}(\Omega)$ is implied by $f\in L^{(p^*)'}(\Omega)$, which is a weaker summability than the one in \cite{L}.

\end{itemize}
\end{remark}

We will give two applications of Theorem \ref{minlocint}. 
The first one deals with nonlinear Cordes conditions for variational problems. Cordes conditions usually refer to the $L^{m}$-theory for elliptic equations in non-divergence form with measurable coefficients, namely, to solutions of
\beq
\label{linsec}
\sum_{ij=1}^{N}a_{ij}D_{ij}u=f
\eeq
with measurable coefficients $a_{ij}:\Omega\subseteq \R^N\to \R$ fulfilling
\[
\lambda\, |\xi|^{2}\le \sum_{ij=1}^{N}a_{ij}(x)\, \xi_{i}\, \xi_{i}\le \Lambda\, |\xi|^{2}
\]
for all $\xi\in \R^{N}$ and a.\,e.\,$x\in \Omega$. Under these measurability assumptions alone, there is no hope in general for the Calder\'on-Zygmund inequality
\beq
\label{czint}
\|u\|_{W^{2, m}(\Omega)}\le C_{m}\,  \|f\|_{ L^{m}(\Omega)},\qquad \forall m>1.
\eeq
Some  regularity has to be assumed on $a_{ij}$ in order to obtain \eqref{czint} for all $m>1$ (VMO regularity suffices; see \cite{CFL}). Roughly stated, a Cordes condition for \eqref{linsec} with discontinuous $a_{ij}$ says that \eqref{czint} holds if either $m$ is sufficiently near $2$, or $\Lambda/\lambda$ is sufficiently near $1$. A similar situation takes place for nonlinear  equations in divergence form.

\begin{theorem}[Theorem \ref{cord}]\label{cordesint}
Let $F$ obey the assumptions of Theorem \ref{minlocint} and $f\in L^{m}(\Omega)\cap W^{-1, p'}(\Omega)$ for some $m>1$. Then any  local minimizer $u\in W^{1, p}(\Omega)$ for $J$ in \eqref{defJ} is such that $DF(Du)\in W^{1, m}_{\rm loc}(\Omega)$ and
\[
\|DF(Du)\|_{W^{1, m}(B_{R})}\le C\, \Big(\|f\|_{L^{m}(B_{2R})}+ \|DF(Du)\|_{L^{m}(B_{2R})}\Big),
\]
for $C=C(K, N, m, R)$ and all $B_{4R}\subseteq \Omega$, in either of the following cases:
\begin{enumroman}
\item
$K\le K_{0}$, with $K_{0}=K_{0}(N, m)>1$.
\item
$|m-2|\le \delta_{0}$, for $\delta_{0}=\delta_{0}(K, N)>0$.
\end{enumroman}
\end{theorem}

As a second application we will derive some partial results pertaining the $C^{p'}$ conjecture, which states that any solution of the $p$-Poisson equation with bounded right-hand side is $C^{p'}$ regular if $p>2$ (hereafter we use the notation $C^{\gamma}=C^{[\gamma], \gamma-[\gamma]}$, being $[\gamma]$ the integer part of $\gamma$).

\begin{theorem}[Corollary \ref{corcor}, Theorem \ref{CZA}, and Corollary \ref{radialcp'}]\label{cpprimoth}
Let $u:\Omega\to \R$ be an approximable solution of \eqref{deltap}.
\begin{enumroman}
\item 
If $f\in L^{\infty}(\Omega)$, then for all sufficiently small $|p-2|$ it holds $u\in C^{2-\alpha_p}(\Omega)$, where $\alpha_p=c(N)\, |p-2|>0$.
\item
For any  $m, p>1$, if $u$ and $\Omega$ have cylindrical symmetry, then $|Du|^{p-2}\, Du\in W^{1, m}_{{\rm loc}}(\Omega)$ whenever $f\in L^{m}_{{\rm loc}}(\Omega)$. In particular, any cylindrical approximable solution of \eqref{deltap} belongs to $ C^{\min\{ p', 2\}-\eps}(\Omega)$ for any $\eps>0$.
\end{enumroman}
\end{theorem}

\begin{remark}\ 

\begin{itemize}
\item
Item (i) confirms the validity of the $C^{p'}$ conjecture near uniform ellipticity, and is inspired by \cite{MRS}. Here, however, we obtain an explicit rate of the H\"older exponent and, more substantially, we do not require Sobolev regularity on $f$. When $f\in L^{\infty}(\Omega)$ the usual notion of weak solution suffices.
\item
Point (ii) gives a weak form of the $C^{p'}$ conjecture (namely, $u\in C^{p'-}$) in the class of cylindrical solutions for any $p\ge 2$, with a different approach than the one of \cite{ATU2}. We need the notion of approximable solutions as in general $f$ may fail to belong to $W^{-1, p'}(\Omega)$ for small $m>1$, destroying the variational setting. For details on such a notion we refer to Section \ref{sola}.

By a cylindrical solution we mean a function of the form $u(x)=v(|x'|)$, $x'\in \R^{k}$, $k\le N$. It is worth noticing that the domain $\Omega$ may not contain the origin, in which case the approach of \cite{ATU2} cannot be easily applied.  
\end{itemize}

\end{remark}

\subsection{Outline of the proofs}

Consider, as a first step, a smooth compactly supported solution $u$ of 
\[
\Div DF(Du)=f \qquad \text{in $\R^{N}$}
\]
with $F$ and  $f$ smooth.  Our starting point is the well-known identity
\beq
\label{divcurl6}
\|DV\|_{L^{2}(\R^{N})}^{2}=\|\Div V\|_{L^{2}(\R^{N})}^{2}+\frac{1}{2}\|\curl V\|_{L^{2}(\R^{N})}^{2} \qquad \forall\,V\in C^{\infty}_{c}(\R^{N}; \R^{N}),
\eeq
where $\curl V= DV-DV^{t}$.
Applying \eqref{divcurl6} to the stress field $V=DF(Du)$, we are reduced to estimate $\curl V$. 

The main observation is that, in the smooth setting, $DV$ is of special type, namely
\[
DV=D^{2}F(Du)\, D^{2}u,
\]
where $D^{2}F(Du)$ is a symmetric positive definite matrix and $D^{2}u$ is symmetric. An elementary lemma shows that any matrix of the form
\[
X=P\, S, \qquad \text{$P$ symmetric positive definite, $S$ symmetric,} 
\]
 fulfills
\beq
\label{gg}
|X-X^{t}|_{2}^{2}\le 2\, \Big(1-\frac{\lambda_{{\rm min}}}{\lambda_{{\rm max}}}\Big)^{2}\, |X|_{2}^{2} \, ,
\eeq
where $\lambda_{{\rm min}}$ and $\lambda_{{\rm max}}$ are, respectively, the minimum and maximum eigenvalues of $P$. Thus the curl term in \eqref{divcurl6} can be reabsorbed to the left if $\lambda_{{\rm max}}\le K\, \lambda_{{\rm min}}$ holds a.\,e.\,for the matrix $D^{2}F$, giving
\[
\|DV\|_{L^{2}(\R^{N})}^{2}\le K^{2}\, \|f\|_{L^{2}(\R^{N})}^{2}
\]
for $V=DF(Du)$ in the smooth, global setting.

\medskip

In order to prove Theorem \ref{minlocint} we have to localize the estimate and to suitably build smooth approximating problems. We regularize the integrand, the source, and the boundary data through convolution but, in order to have strongly elliptic problems, we would like to add a small multiple of a strongly elliptic functional. Since we do not want to alter the q.\,u.\ convexity constant, the only viable choice is to add small multiples of $|z|^2$ to the regularized integrands. This is a quite unnatural choice if $F$ is not of standard quadratic growth, and it forces some interplay between the regularization parameters. Here, the explicit {\em a-priori} Lipschitz estimate for the corresponding solutions taken from \cite{BB} plays a key role.
 
\medskip

The main step of the proof of Theorem \ref{cordesint} is to represent the solutions of 
\[
\begin{cases}
\Div V=f\\
\curl V=G
\end{cases}
\]
in $\R^{N}$ through Riesz transforms and generalize \eqref{divcurl6} to the $L^{m}$-setting as
\[
\|DV\|_{L^{m}(\R^{N})}\le C(m, N)\, \big(\|\Div V\|_{L^{m}(\R^{N})}+\|\curl V\|_{L^{m}(\R^{N})}\big).
\]
Thus, inequality \eqref{gg} does the trick in case (i) of Theorem \ref{cordesint}, allowing reabsorption of the $\curl$ term for $K$ sufficiently near $1$. A standard Riesz interpolation argument, together with a careful choice of the norms involved, allows to prove case (ii).

\medskip

Finally, point (i) of Theorem \ref{cpprimoth} is an immediate consequence of point (i) of Theorem \ref{cordesint}, while point (ii) stems from this observation: if $u$ exhibits cylindrical symmetry, then the stress field $a(|Du|)\, Du$ is irrotational; by the Helmoltz decomposition, it can be locally represented as the gradient of a solution of $\Delta v=f\in L^{m}$, so the standard Calder\'on-Zygmund theory applies.

\subsection{Structure of the paper}
In section \ref{prel} we recall some functional inequalities and properties about quasiuniform convexity. In section \ref{3} we develop our basic estimate in the smooth setting; section \ref{locmins} is devoted to the proof of Theorem \ref{minlocint}, where the main approximation procedure, used also later, is described; section \ref{examples} contains the relevant examples depicted above. In section \ref{5} we focus on  the applications: first we treat the Cordes conditions, and finally we collect the partial results pertaining the $C^{p'}$ conjecture.

\medskip

{\bf Notations:}
\begin{itemize}
\item[-] The euclidean norm of a vector $v\in \R^{N}$ is denoted by $|v|$ and $(v, w)$ denotes the scalar product. By $\Omega$ we denote a bounded open subset of $\R^{N}$, while $B_{r}$ denotes a ball of radius $r$ not necessarily centered at the origin. Similarly, $B$ denotes a ball with unspecified center and radius; if $B=B_{r}(x_{0})$, then we set $\lambda B=B_{\lambda r}(x_{0})$.
\item[-] For a $N\times N$ matrix $A=(a_{ij})$, its transpose is denoted by $A^{t}$; on such matrices we consider the Frobenius norm
\[
|A|_{2}=\Big(\sum_{i, j=1}^{N}|a_{ij}|^{2}\Big)^{1/2},
\]
arising from the scalar product $(A, B)_2={\rm Tr}\, (A\, B^t)$. For $v, w\in \R^{N}$, $v\otimes w$ is the matrix with entries $(v_{i}\, w_{j})$, while $v\land w=v\otimes w-w\otimes v$; $I$ denotes the identity matrix, while ${\rm Id}$ the identity function of $\R^N$; $O_{N}$ is the orthogonal group of $N\times N$ matrices such that $A\, A^{t}=A^{t}\, A=I$; if $A$ is symmetric, $\lambda_{{\rm min}}(A)$, $\lambda_{{\rm max}}(A)$ are its minimum and its maximum eigenvalues. 
\item[-] If a domain of integration is missing, it means that we are integrating on $\R^{N}$. We also set for brevity $\|f\|_{m}=\|f\|_{L^{m}(\R^{N})}$. Finally, we sometimes will use the $L^{m}$ ``norm" also for $m\in (0, 1)$, when it is only positively $1$-homogeneous. 
\item[-] As customary, we indicate with $ W^{-1,p'}(\Omega) $ the dual space of $ W^{1,p}_0(\Omega) $.
\end{itemize}
\section{Preliminaries}\label{prel}

\subsection{Functional inequalities}
For $0<\theta<m$, $m>1$, starting from the inequality
\beq
\label{sob}
\int_{B_{1}} |v|^{m}\, dx\le \int_{B_{1}}|Dv|^{m}\, dx+C(N, m, \theta)\left(\int_{B_{1}} |v|^{ \theta}\, dx\right)^{m/\theta}, 
\eeq
(obtained by a standard compactness argument) and replicating, with the obvious modifications, the proof of \cite[eq. (5.4)]{CM}, we get the following functional inequality.
\begin{lemma}\label{CM}
Let $m>1$, $0<\theta<m$, and $ R \le r < s < 2R $. There exists $C=C(N, m, \theta)$ such that for any $v\in W^{1,m}(B_{s}\setminus B_{r})$ and any $\delta\in (0, 1)$  it holds
\[
\int_{B_{s}\setminus B_{r}} |v|^{m}\, dx\le \delta^{m}\, R^{m}\int_{B_{s}\setminus B_{r}} |Dv|^{m}\, dx+\frac{C}{(\delta^{N}\, (s-r)\, R^{N-1})^{\frac{m-\theta}{\theta}}}\left(\int_{B_{s}\setminus B_{r}} |v|^{\theta}\, dx\right)^{m/\theta}.
\]
\end{lemma}

The following lemma is a straightforward generalization of the well-known identity
\beq
\label{asl}
 \int  |DV|_{2}^{2}\, dx= \frac{1}{2}\int |\curl V|_{2}^{2}\, dx+\int (\Div V)^{2}\, dx,
 \eeq
 valid for $V\in C^{1}_{c}(\R^{N}; \R^{N})$.
 
\begin{lemma}\label{divrotl2}
If $V\in C^{2}(\R^{N}, \R^{N})$, then for any $\varphi\in C^{2}_{c}(\R^{N})$ it holds
\beq
\label{dr}
\begin{split}
\int\varphi^{2}\, |DV|_{2}^{2}\, dx&= \frac{1}{2}\int\varphi^{2}\, |\curl V|_{2}^{2}\, dx+\int\varphi^{2}\, (\Div V)^{2}\, dx\\
& \quad +\int \left[2\, (D\varphi^{2}, V)\, \Div V+(D^{2}\varphi^{2}, V\otimes V)_{2}\right]\, dx.
\end{split}
\eeq
\end{lemma}

\begin{proof}
Write, through parallelogram identity,
\[
|DV|_{2}^{2}=\frac{1}{4}|DV-DV^{t}|_{2}^{2}+\frac{1}{4}|DV+DV^{t}|_{2}^{2},
\]
then multiply by $\varphi^{2}$ and integrate to obtain
\beq
\label{rf}
\begin{split}
\int \varphi^{2}\,  |DV|^{2}_{2}&\, dx=\int \frac{\varphi^{2}}{4}\, |\curl V|_{2}^{2}\, dx+\int \frac{\varphi^{2}}{4}\textstyle\sum_{ij}\displaystyle(D_{j}V^{i}+D_{i}V^{j})^{2}\, dx\\
&=\int \frac{\varphi^{2}}{4}|\curl V|_{2}^{2}\, dx+\int \frac{\varphi^{2}}{2}\, |DV|_{2}^{2}\, dx+\int \frac{\varphi^{2}}{2}\, \textstyle \sum_{ij}D_{i}V^{j}D_{j}V^{i}\, dx.
\end{split}
\eeq
The last term is computed integrating by parts twice: for any $i, j=1, \dots, N$
\[
\begin{split}
\int \varphi^{2}\, D_{i}V^{j}D_{j}V^{i}\, dx&=-\int D_{i}\varphi^{2}\, V^{j}D_{j}V^{i}\, dx-\int \varphi^{2}\, V^{j}D_{ij}V^{i}\, dx\\
&=\int D_{ij}\varphi^{2}\, V^{j}V^{i}\, dx+\int D_{i}\varphi^{2}\, V^{i}\, D_{j}V^{j}\, dx\\
&\quad +\int D_{j}\varphi^{2}\, V^{j}D_{i}V^{i}\, dx+\int \varphi^{2}\, D_{j}V^{j}D_{i}V^{i}\, dx,
\end{split}
\]
so that summing over $i, j$ gives
\[
\int \varphi^{2}\, \textstyle \sum_{ij}\displaystyle D_{i}V^{j}D_{j}V^{i}\, dx=
\int \left[\varphi^{2}\, (\Div V)^{2}+2\,(D\varphi^{2}, V)\, \Div V+(D^{2}\varphi^{2}, V\otimes V)_{2}\right]\, dx. 
\]
Inserting this formula into \eqref{rf} yields \eqref{dr}.

\end{proof}

\subsection{Quasiuniform convexity}

In this section we show that the q.\,u.\,convexity condition provides, in a unified way, many of the properties that the usual integrands of the Calculus of Variations satisfy.

Let us begin by observing  that the gradient  of a convex $F$ is defined a.\,e.\, and belongs to $BV_{\rm loc}(\R^N)$ (see \cite{AA}). Accordingly, the second derivative  of $F$ can be decomposed in an absolutely continuous part, a jump part, and a Cantor part. If $F\in C^1(\R^N)$ the jump part vanishes, hence by requiring that $F\in C^1(\R^N)\cap W^{2, 1}_{\rm loc}(\R^N)$ we are actually excluding that $D^2F$ has a Cantor part.

Now we discuss in detail some consequences of the q.\,u.\,convexity condition; although condition (iv) of Proposition \ref{proqu} will be not used in the sequel, we prove it for the sake of completeness.

\begin{proposition}[Properties of $K$-quasiuniformly convex functions]\label{proqu}
Let $F$ be a $K$-quasiuniformly convex function. Then
\begin{enumroman}
\item
$DF$ is $K^{N-1}$-quasiconformal, hence $C^{1/K}(\R^N)$.
\item
$F$ is strictly convex and of $(p, q)$-growth, i.\,e.\,, there exists $C=C(N, K, F)>0$ and  $1<p<q<+\infty$  such that 
\beq
\label{pqg}
C^{-1}|z|^p-C \le F(z)\le C\, \big(|z|^q+1\big)
\eeq
\beq
\label{pqg2}
C^{-1}|z|^{p-1}-C \le |DF(z)|\le C\, \big(|z|^{q-1}+1\big)
\eeq
for all $z\in \R^N$. More precisely, one can take $p=1+1/K$ and $q=1+K$.
\item
If $\varphi\in C^\infty_c(\R^N; [0, +\infty))$, then $F*\varphi$ is $K$-q.\,u.\,c..
\item
Its Moreau-Yoshida regularization
\beq
\label{MY}
F_{\delta}(z)= \inf_{y\in \R^{N}}\Big\{F(y)+\frac{1}{2\, \delta} |y-z|^{2}\Big\}
\eeq
is $K$-q.\,u.\,c.. 
\end{enumroman}
\end{proposition}

\begin{proof}
\begin{enumroman}
\item
By the Alexandrov theorem $DF$ is differentiable a.\,e.\,, and \cite[Theorem 32.3]{V} ensures that $DF\in W^{1,N}_{\rm loc}(\R^N)$. Since \eqref{qcd} and \eqref{Kell} are equivalent up to changing the constants, $u$ is quasiuniformly convex in the sense of \cite{KoMa}. In particular \cite[Theorem 3.1]{KoMa} shows that $DF$ is $K^{N-1}$-quasiconformal. The regularity statement holds for any quasiconformal mapping (see \cite[Theorem 2.14]{Ma}).
\item
The strict convexity of $F$ follows from \cite[Lemma 3.2]{KoMa}. If $z_0$ is the unique minimum point for $F$ we can consider $F(z_0+\cdot)-F(z_0)$, so there is no loss in generality assuming $DF(0)=0$, $F(0)=0$. Let $G=DF$, which then is  $K^{N-1}$-quasiconformal. By \cite[Theorem 2.14]{Ma}
\[
|G(z)|\le C(N, K)\, \sup_{y\in B_1} |G(y)| \, |z|^{1/K},\qquad z\in B_1.
\]
Since $G^{-1}$ is still $K^{N-1}$ quasiconformal, it obeys a similar estimate, proving the lower bound
\[
|G(z)|\ge C(N, K, G) \, |z|^{K}, \qquad z\in B_1. 
\]
Finally, the inversion 
\[
G^*(x)=\frac{G(x/|x|^2)}{|G(x/|x|^2)|^2}
\]
is again $K^{N-1}$ quasiconformal on $B_1$, so that the previous estimates are transferred to the outside of $B_1$ as 
\beq
\label{K}
C^{-1}\, |z|^{1/K} \le |G(z)|\le C\, |z|^K \qquad |z|\ge 1,
\eeq
where $C=C(N, K, G)$. For $G=DF$,  $p=1+1/K$, $q=1+K$ we thus obtained \eqref{pqg2}. Moreover, by \cite[Lemma 18]{K}, $DF$ is $\delta$-convex for some $\delta=\delta(N, K)>0$, meaning that 
\beq
\label{deltc}
(DF(z)-DF(y), z-y)\ge \delta\, |DF(z)-DF(y)||z-y| \qquad \forall z, y\in \R^N.
\eeq

Using \eqref{deltc} and \eqref{pqg2}, we get
\[
F(z)=\int_0^1(DF(t\, z), z)\, dt\ge \delta \int_0^1|DF(t\, z)|\, |z|\, dt \ge \frac{\delta}{p\, C}\, |z|^{p} - C\, |z| \ge \frac{\delta}{p\, C}\, |z|^{p} - C,
\]
by  sufficiently increasing $ C $ in the last inequality. This produces the lower bound in \eqref{pqg}, while the upper bound follows from  \eqref{pqg2} alone through a similar calculation.

\item
Let $\lambda_{\rm min}(z)=\lambda_{\rm min}(D^2F(z))$, where $z$ is a second order differentiability point for $F$. From the representation
\[
\lambda_{\rm min}(z):=\inf\big\{(D^{2}F(z)\,\xi, \xi):\xi\in D\big\},
\]
where $D$ is a fixed countable dense subset of $\mathbb{S}^{N-1}$, we infer that $\lambda_{{\rm min}}$ is measurable and in $L^1_{\rm loc}(\R^N)$. Then, for any $z\in \R^{N}$ and $\xi\in D$,
\[
\begin{split}
(D^{2}F*\varphi (z)\, \xi, \xi)&=\int\varphi(z-y)\, (D^{2}F(y)\, \xi, \xi)\, dy\\
&\ge \int\varphi(z-y)\, \lambda_{{\rm min}}(y)\, |\xi|^{2}\, dx=:\tilde{\lambda}_{{\rm min}}(z)\, |\xi|^{2},
\end{split}
\]
while
\[
\int\varphi(z-y)\, (D^{2}F(y)\, \xi, \xi)\, dy\le \int\varphi(z-y)\, K\,  \lambda_{{\rm min}}(y)\, |\xi|^{2}\, dy=K\, \tilde{\lambda}_{{\rm min}}(z)\, |\xi|^{2},
\]
implying the claim.
\item
Recall that the minimum in \eqref{MY} is attained at a unique point $P_{\delta}(z)$ fulfilling
\beq
\label{proj}
P_\delta(z)+\delta\, DF(P_{\delta}(z))=z, \qquad DF_{\delta}(z)=DF(P_{\delta}(z)),
\eeq
and the so-defined function $P_{\delta}=({\rm Id}+\delta\, DF)^{-1}$ is 1-Lipschitz and a homeomorphism of $\R^{N}$, since $F\in C^{1}(\R^{N})$. Let $E$ be the set of points where $DF$ fails to be differentiable. The map ${\rm Id}+\delta\, DF$ is the gradient of a  $K$-quasiuniformly convex function, hence by point (i) is quasiconformal and satisfies the Lusin $(N)$ property, i.\,e.\,, it sends null-measure sets to null-measure sets (see \cite{V}). Thus, $P_\delta^{-1}(E)$ has zero measure. Moreover, since $P_\delta$ is Lipschitz continuous, Rademacher's theorem ensures that the set $M$ where $P_\delta$ is not differentiable has zero measure. We will prove \eqref{Kell} at any point 
\[
z\notin M\cup P_\delta^{-1}(E),
\]
the latter set having zero measure. Indeed, at any such point $z$ we have that $DF$ is differentiable at $P_\delta(z)$ and $P_\delta$ is differentiable at $z$. The Chain Rule applied to \eqref{proj} then gives
\[
\big(I+\delta\, D^2F(P_\delta(z))\big)\, DP_\delta(z)=I,\qquad D^2F_\delta(z)=D^2F(P_\delta(z))\, DP_\delta(z),
\]
which entails
\beq
\label{D2F}
D^{2}F_{\delta}(z)=D^{2}F(P_{\delta}(z))\, \big(I+\delta\, D^{2} F(P_{\delta}(z))\big)^{-1}.
\eeq
Let  the eigenvalues of $D^{2}F(P_\delta(z))$ be $\lambda_{\rm min} = \lambda_{1}\le \dots \le \lambda_{N} = \lambda_{\rm max}$. The matrices $D^{2}F(P_\delta(z))$ and $(I+\delta\, D^{2}F(P_\delta(z)))^{-1}$ have the same basis of eigenvectors, with eigenvalues $\lambda_{i}$ and $(1+\delta\, \lambda_{i})^{-1}$ respectively.
Hence  \eqref{D2F} implies that $D^{2}F_{\delta}(z)$ has eigenvalues  $\lambda_{i}/(1+\delta\, \lambda_{i})$.
As $t\mapsto t/(1+\delta\, t)$ is increasing, its minimum and maximum eigenvalues are
\[
\lambda_{\delta, \rm min}:=\frac{\lambda_{{\rm min}}}{1+\delta\, \lambda_{{\rm min}}},\qquad \lambda_{\delta, \rm max}:=\frac{\lambda_{{\rm max}}}{1+\delta\, \lambda_{{\rm max}}},
\]
which obey $\lambda_{\delta, \rm max}\le K\, \lambda_{\delta, \rm min}$ as long as $\lambda_{\rm max}\le K\, \lambda_{\rm min}$.

\end{enumroman}
\end{proof}
Due to the previous proposition, we will denote  henceforth by $p$ and $q$ the power of the lower and upper bound, respectively, for a given $K$-q.\,u.\,c.\,function $F$.
\subsection{Extensions}
We conclude with a couple of tools which will be occasionally used in the following. 

\begin{lemma}\label{locality}
Let $F\in C^{1}(B_{R})$ be a non-negative, strictly convex function, and $ \sigma \in (0,1) $.
\begin{enumroman}
\item
 There exists a strictly convex $\tilde{F}\in C^{1}(\R^{N})$ such that $\tilde{F}\big|_{B_{\sigma R}}=F$ and, for some $C$, $\alpha$ depending on $F$, $R$, $N$, $\sigma$, it holds
\beq
\label{co}
|z|^2\le \alpha(\tilde{F}(z)+1),\qquad C^{-1}{|z|}-C\le |D\tilde{F}(z)|\le C\, (|z|+1).
\eeq
\item
If $F$ is $K$-q.\,u.\,c.\, in $B_{R}$ as per Definition \ref{defKell}, and for some $\eps>0$ it holds $\lambda_{\rm min}(z) \ge \eps $ in $B_{R}\setminus B_{\sigma R}$, in addition to \eqref{co}  $\tilde{F}$ can be chosen to be $\tilde{K}$-q.\,u.\,c.\,, with $\tilde{K}=\tilde{K}(F, R, N, \sigma, \eps)$.
\end{enumroman}
\end{lemma}

\begin{proof}
To prove (i), let $ \tau= (1+\sigma)/2 $, choose a radial cut-off function $\eta\in C^{\infty}_{c}(B_{R}; [0, 1])$ such that $\eta\equiv 1$ in $B_{\tau R}$, and define
\[
\tilde{F}(z)=\eta(z) \, F(z)+(1-\eta(z))\, \frac{|z|^{2}}{2}+ C\, (|z|- \sigma\,  R)_{+}^{2},
\]
where $C>0$ is a constant to be chosen. Clearly $\tilde{F}\in C^{1}(\R^{N})$ and obeys \eqref{co}, so it remains to show that $F$ is strictly convex for a suitable $C$. To this aim, set
\[
A(z) := \frac{1}{2} \, D^2 (|z|-\sigma R)^2 = \frac{\sigma R}{|z|} \, \frac{z}{|z|} \otimes \frac{z}{|z|} + \Big(1-\frac{\sigma R}{|z|}\Big) I,
\]
whose eigenvalues are $ 1 $ and $ 1-\sigma R/|z| $. In particular, $A$ is non-negative definite in $B_{R} \setminus B_{\sigma R}$, and its eigenvalues are uniformly bounded below in $\R^N\setminus B_{\tau R}$ by a positive constant; since $\tilde{F}$ agrees with $F$ in $B_{\sigma R}$, it follows that $\tilde{F}$ is strictly convex in $B_{\tau R}$ and in $\R^N\setminus B_R$. A straightforward computation yields
\[
\begin{split}
D^{2}\tilde{F}&=\eta\, D^{2}F+M+2\, C A\\
M&:=(1-\eta)\, I+ D\eta\otimes DF +DF\otimes D\eta -2\, D\eta\otimes z+\Big(F-\frac{|z|^2}{2}\Big)\, D^{2}\eta
\end{split}
\]
a.\,e.\, outside $ B_{\sigma R} $, and we can choose $C$ so that
\[
\Big(1-\frac{\sigma}{\tau}\Big) C=\max_{z\in B_{R}}|M(z)|_{2},
\]
ensuring 
\[
\lambda_{{\rm min}}(D^{2}\tilde{F})\ge \eta\, \lambda_{{\rm min}}(D^{2}F)+ \Big(1-\frac{\sigma}{\tau}\Big) C\qquad \text{ in $B_R\setminus B_{\tau R}$}.
\]
Summing up, $\tilde{F}$ is globally strictly convex by an elementary argument.

To prove (ii), let $\tilde{\lambda}_{{\rm min}}(z)$ and $\tilde{\lambda}_{{\rm max}}(z)$ denote the minimum and maximum eigenvalues of $D^{2}\tilde{F}(z)$, and $\lambda_{{\rm min}}, \lambda_{\rm max}$ those for $D^{2}F$. It holds
\[
\eps\le \lambda_{{\rm min}} \le \tilde{\lambda}_{{\rm min}} \le \tilde{\lambda}_{{\rm max}}\le 3\, C+ \lambda_{{\rm max}}
\]
 in $B_{\tau R}\setminus B_{\sigma R}$, so that from \eqref{Kell} we get
\[
\frac{\tilde{\lambda}_{{\rm max}}}{\tilde{\lambda}_{\rm min}}\le\frac{\tilde{\lambda}_{{\rm max}}}{\lambda_{\rm min}} \le \frac{3\, C}{\lambda_{{\rm min}}}+K\le \frac{3\, C}{\eps}+K.
\]
In $\R^{N}\setminus B_{\tau R}$ it holds
\[
(1-\sigma/\tau)\, C+\eta\, \lambda_{{\rm min}} \le \tilde{\lambda}_{{\rm min}} \le \tilde{\lambda}_{{\rm max}}\le 3\, C+\eta\, \lambda_{{\rm max}},
\]
so that
\[
\frac{\tilde{\lambda}_{{\rm max}}}{\tilde{\lambda}_{\rm min}}\le\frac{ 3\, C+ K\, \eta\, \lambda_{{\rm min}}}{(1-\sigma/\tau)\, C+\eta\, \lambda_{{\rm min}}}\le \frac{3\, \tau}{\tau-\sigma} + K.
\]
Since $\tilde{F}=F$ on $B_{\sigma R}$ and \eqref{Kell} holds for $F$ there, the claim is proved.
\end{proof}

\section{Divergence form, quasiconformal equations}

\subsection{The smooth setting}\label{3}

The core of our approach lies in the following elementary observation.

\begin{lemma}\label{basiclemma}
Let $X=P\, S$, where $P$ and $S$ are symmetric  $N\times N$ matrices and $P$ is positive definite with minimum and maximum eigenvalues $\lambda_{{\rm min}}$ and $\lambda_{{\rm max}}$. Then 
\beq
\label{matrixin}
|X-X^{t}|_{2}^{2}\le 2\, \frac{(1-\lambda_{{\rm min}}/\lambda_{{\rm max}})^{2}}{1+(\lambda_{{\rm min}}/\lambda_{{\rm max}})^{2}}\, |X|_{2}^{2}.
\eeq
\end{lemma}

\begin{proof}
Inequality \eqref{matrixin} is invariant under rotations; thus, without loss of generality, we can suppose $P_{ij}=\lambda_{i}\delta_{ij}$ with $0<\lambda_{1}\le \dots \le \lambda_{N}$. Then from $X=P\, S$ we get
\[
X_{ij}=\lambda_i\, S_{ij},
\]
so that
\[
|X-X^{t}|_{2}^{2}=\sum_{ij}|X_{ji}-X_{ij}|^{2}=2\,\sum_{i<j}|X_{ji}-X_{ij}|^{2}=2\, \sum_{i<j}|\lambda_j\, S_{ji}-\lambda_i\, S_{ij}|^{2},
\]
and from the symmetry of $S$ we conclude
\beq
\label{l1}
|X-X^{t}|_{2}^{2}=2\, \sum_{i<j}S_{ij}^2\, (\lambda_j-\lambda_i)^{2}.
\eeq
Similarly, we have
\beq
\label{l2}
|X|_{2}^{2}\ge\sum_{i<j}\big(X_{ij}^{2}+X_{ji}^{2}\big)= \sum_{i<j} S_{ij}^2(\lambda_i^2+\lambda_j^2).
\eeq
Let 
\[
\varphi (t)=\frac{(1-t)^2}{1+t^2},
\]
which is decreasing in $[0, 1]$, and observe that for $j>i$ we have $\lambda_i/\lambda_j\in [0, 1]$. Therefore 
\[
(\lambda_j - \lambda_i)^2=\frac{(\lambda_j- \lambda_i)^2}{\lambda_j^2+\lambda_i^2}\, (\lambda_j^2+\lambda_i^2)=\varphi\Big(\frac{\lambda_i}{\lambda_j}\Big)\, (\lambda_j^2+\lambda_i^2)\le \varphi\Big(\frac{\lambda_{\rm min}}{\lambda_{\rm max}}\Big)(\lambda_j^2+\lambda_i^2).
\]
Inserting this estimate in \eqref{l2} and recalling \eqref{l1} we get
\[
|X-X^{t}|_{2}^{2}\le 2\,   \varphi\Big(\frac{\lambda_{\rm min}}{\lambda_{\rm max}}\Big) \sum_{i<j} S_{ij}^2(\lambda_i^2+\lambda_j^2)=2\,   \varphi\Big(\frac{\lambda_{\rm min}}{\lambda_{\rm max}}\Big)\, |X|_{2}^{2} .
\]
\end{proof}

\begin{theorem}\label{Mth}
Let $u\in C^{2}(B_{2R})$ solve 
\[
\Div(DF(Du))=f\qquad \text{in $B_{2R}$}
\]
for a $K$-q.\,u.\,c.\,$F\in C^2(\R^N)$, and let 
\[
V(x)=DF(Du(x)).
\] 
Then for any $\theta\in (0, 2]$ there exist $C=C(N, K, \theta)$ and $C_{R}=C(N, K, \theta, R)$ such that 
\beq
\label{w12est}
\|V\|_{W^{1,2}(B_{R})}\le C\, \|f\|_{L^{2}(B_{2R})} + C_{R}\, \|V\|_{L^{\theta}(B_{2R})}.
\eeq
\end{theorem}

\begin{proof}
For any $\eps>0$ let
\[
F_{\eps}(z)=F(z)+\eps\, \frac{|z|^{2}}{2}, \qquad f_{\eps}=f+\eps\, \Delta u,
\]
so that $D^{2}F_{\eps}$ is symmetric and positive definite. It holds
\[
\lambda_{{\rm min}}(D^{2}F_{\eps}(z))=\lambda_{{\rm min}}(D^{2}F(z))+\eps, \qquad \lambda_{{\rm max}}(D^{2}F_{\eps}(z))=\lambda_{{\rm max}}(D^{2}F(z))+\eps,
\]
so that if \eqref{Kell} holds for $F$, it does so  for $D^{2}F_{\eps}$ as well, with the same constant $K$.

Clearly $u$ solves 
\[
\Div (DF_{\eps}(Du))=f_{\eps}
\]
in $B_{2R}$. Letting
\[
V_\eps=DF_{\eps}(Du),
\]
it holds
\[
D V_\eps=D^{2}F_{\eps}(Du)\, D^{2}u,
\]
where the first matrix is symmetric positive definite and the second one is symmetric. We thus apply Lemma \ref{basiclemma}   to the matrix 
\[
P=D^2F_\eps(Du),\qquad S=D^2u, \qquad X=DV_\eps=P\, S.
\]
Recall that $D^2F_\eps(Du)$ fulfills \eqref{Kell} with constant $K$, whence
\[
\frac{\big({\lambda_{{\rm max}}(D^2F_\eps(Du))-\lambda_{\rm min}}(D^2F_\eps(Du))\big)^{2}}{\lambda_{{\rm max}}^2(D^2F_\eps(Du))+\lambda_{{\rm min}}^{2}(D^2F_\eps(Du)) } \le  \frac{(K-1)^2}{K^2+1}. 
\]
From \eqref{matrixin} we get
\beq
\label{in1}
|\curl V_\eps|_{2}^{2}\le 2\, \frac {(K-1)^2}{K^2+1}\, |DV_\eps|_{2}^{2}.
\eeq
 
 For any $r, s$ with $R\le r<s\le 2\, R$, fix $\varphi\in C^{\infty}_{c}(B_{s}, [0, 1])$ such that 
 \beq
 \label{propvarphi}
 \left.\varphi\right|_{B_{r}}\equiv 1, \qquad |D\varphi|\le \frac{C}{s-r},\qquad |D^{2}\varphi|\le \frac{C}{(s-r)^{2}}.
 \eeq
This will allow to consider $\varphi\, V_\eps$ as defined on the whole $\R^{N}$, so that \eqref{dr} holds true. The stipulated properties of $\varphi$ ensure that 
\[
\int 2\, (D\varphi^{2}, V_\eps)\, f_{\eps}+(D^{2}\varphi^{2}, V_\eps\otimes V_\eps)_{2}\, dx
\le \frac{C}{(s-r)^{2}}\int_{B_{s}\setminus B_{r}}|V_\eps|^{2}\, dx+C\int_{B_{2R}}f_{\eps}^{2}\, dx,
\]
where we used the Schwartz inequality on the first term and $s\le 2R$. Using also \eqref{in1} to control the curl term of Lemma \ref{divrotl2} yields
\[
\int \varphi^{2}  |DV_\eps|_{2}^{2}\, dx\le \Big(1-\frac{1}{K}\Big)^2 \int \varphi^{2}\, |DV_\eps|_{2}^{2} \, dx+\frac{C}{(s-r)^{2}}\int_{B_{s}\setminus B_{r}}|V_\eps|^{2}\, dx+C\int_{B_{2R}}f_{\eps}^{2}\, dx.
\]
We let $\eps\to 0$ and bring the first term on the right to the left-hand side; recalling that $\varphi\equiv 1$ on $B_{r}$, we obtain
\beq
\label{po}
\int_{B_{r}}|DV|_{2}^{2}\, dx\le \frac{C_{K}}{(s-r)^{2}}\int_{B_{s}\setminus B_{r}}|V|^{2}+ C_{K}\int_{B_{2R}}f^{2}\, dx
\eeq
for any $R\le r<s\le 2R$.
We next proceed as in \cite{CM}: by Lemma \ref{CM} with $m=2$ and 
\[
\delta=\frac{s-r}{2\, \sqrt{C_{K}}\, R}
\] 
we get
\[
\frac{C_{K}}{(s-r)^{2}}\int_{B_{s}\setminus B_{r}}|V|^{2}\, dx\le \frac{1}{4}\int_{B_{s}\setminus B_{r}}|DV|_{2}^{2}\, dx+ \frac{C_{K}\, R^{\frac{2-\theta}{\theta}}}{(s-r)^{2+\frac{2-\theta}{\theta}(N+1)}}\left(\int_{B_{s}\setminus B_{r}}|V|^{\theta}\, dx\right)^{2/\theta}
\]
which inserted into \eqref{po} gives, for all $R\le r<s\le 2R$,
\[
\int_{B_{r}}|DV|_{2}^{2}\, dx\le \frac{1}{4}\int_{B_{s}}|DV|_{2}^{2}\, dx+C_{K}\int_{B_{2R}}f^{2}\, dx +\frac{C_{K}\, R^{\frac{2-\theta}{\theta}}}{(s-r)^{2+\frac{2-\theta}{\theta}(N+1)}}\left(\int_{B_{s}\setminus B_{r}}|V|^{ \theta}\, dx\right)^{2/\theta}.
\]
A standard iteration lemma (see \cite[Lemma 3.1, Ch 5]{G}) improves the latter to
\[
\int_{B_{R}}|DV|_{2}^{2}\, dx\le C_{K}\int_{B_{2R}}f^{2}\, dx +\frac{C_{K}}{R^{2+\frac{2-\theta}{\theta}N}}\left(\int_{B_{2R}}|V|^{\theta}\, dx\right)^{2/\theta},
\]
which is the desired estimate on the derivative of $V$. In order to control $\|V\|_{L^{2}(B_{R})}^{2}$, we invoke the rescaled form of \eqref{sob} which, in conjunction with the previous estimate, completes the proof of \eqref{w12est}.
\end{proof}

\subsection{Local minimizers}\label{locmins}

For a bounded $\Omega\subseteq \R^{N}$ we let
\[
J(w, \Omega)=\int_{\Omega} F(Dw)\, dx+\int_{\Omega} f\, w\, dx
\]
whenever the two integrands are in $L^{1}(\Omega)$, sometimes omitting the dependence on $\Omega$ when this causes no confusion. We will consider  $J$ under $p$-coercivity assumptions on $F$ and for  $f\in W^{-1, p'}(\Omega)$, so that it is  well defined on $W^{1,p}(\Omega)$.

Recall that $u\in W^{1,p}_{{\rm loc}}(\Omega)$ is a local minimizer for $J$ in $W^{1, p}(\Omega)$ if, for any $B\Subset \Omega$,
\beq
\label{defj2}
J(u, B)=\inf\big\{J(w, B): w\in u+ W^{1,p}_{0}(B)\big\}.
\eeq

\begin{theorem}\label{minloc}
Let $F\in C^{1}(\R^{N})$ be a q.\,u.\,c.\,function and $q>p>1$ be given in Proposition \ref{proqu}, point (ii). For $f\in L^{2}(\Omega)\cap W^{-1, p'}(\Omega)$, let $u$ be a local minimizer $u$ of $J$ in $\Omega$. Then, for any ball $B$ such that $4 B\subseteq \Omega$ it holds $DF(Du)\in W^{1,2}(B)$, 
\beq
\label{mainest2}
\|DF(Du)\|_{L^2(B)}\le C\Big(1+\|f\|_{L^{2}(2 B)}+\|F(Du)\|_{L^{1}(2 B)}^{(q-1)/p}\Big)
\eeq
for  $C=C(K, N, B)>0$, and for any $\theta\in (0, 2]$
\beq
\label{mainest}
\|DF(Du)\|_{W^{1,2}(B)}\le C\Big(\|f\|_{L^{2}(2 B)}+\|DF(Du)\|_{L^{\theta}(2 B)}\Big)
\eeq
for $C=C(K, N, B, \theta)>0$.
 Moreover, $u$ satisfies the Euler-Lagrange equation
\beq
\label{EL}
\int_\Omega (DF(Du), D\varphi)\, dx=\int_\Omega f\, \varphi\, dx\qquad \forall \varphi\in C^\infty_c(\Omega).
\eeq
\end{theorem}

\begin{proof}
Let ${\rm Argmin} (F)=\{z_0\}$. By considering $\tilde{F}(z)=F(z+z_{0})-F(z_{0})$ and $\tilde{u}(x)=u(x)-(z_{0},  x)$ and noting that $f(\cdot)\, (z_{0}, \cdot )\in L^{1}_{{\rm loc}}(\Omega)$, it is readily checked that $\tilde u$ turns out to be a  local minimizer of 
\[
\tilde{J}(w, \Omega)=\int_{\Omega} \tilde{F}(Dw) +f\, w\, dx.
\]
Hence, hereafter we suppose $F(z)\ge F(0)=0$. Finally, recalling Proposition \ref{proqu}, point (ii), we know that $F$ is strictly convex and 
\beq
\label{fdf}
C^{-1}|z|^p-C\le F(z)\le C\,\big( |z|^q+1),\qquad |DF(z)|\le C\, (|z|^{q-1}+1),
\eeq
so that $u$ is the unique minimizer locally, with respect to its own boundary values.
We split the proof in several steps.

\medskip
{\em Step 1} (Approximating problems)\\
Fix $\varphi\in C^{\infty}_{c}(B_{1}, [0, +\infty[)$ such that $\|\varphi\|_{1}=1$ and let $\varphi_{\eps}(x)=\eps^{-N}\, \varphi(x/\eps)$. 
For  $B\Subset \Omega$ and $n\in \N$, let $f_{n}=f*\varphi_{\frac{1}{n}}$ and, for $\eps_{n}, \mu_{n}\to 0^{+}$ to be chosen, 
\[
J_{n}(w)=\int_{B}F*\varphi_{\eps_{n}}(Dw)+\frac{\mu_{n}}{2}\, |Dw|^{2} +f_{n}\, w\, dx.
\]
Set 
\[
{\rm Lip}_{\psi}(B)=\{w\in {\rm Lip}(\overline{B}):w=\psi\ \text{on $\partial B$}\}.
\]
According to \cite[Theorem 9.2]{S}, there is a solution $v_{n}\in {\rm Lip}_{u*\varphi_{\frac{1}{n}}}(B)$ of 
\[
J_{n}(v_{n})=\inf\{J_{n}(w):w\in {\rm Lip}_{u* \varphi_{\frac{1}{n}}}(B)\},
\]
since $u*\varphi_{\frac{1}{n}}$ is smooth on $\partial B$ (thus satisfying the Bounded Slope Condition) and also $f_{n}$ is smooth. Moreover, for any {\em fixed} $n$ there is no Lavrentiev gap for $J_{n}$; see \cite[p. 5923]{BMT}. Hence $v_{n}$ also solves 
\[
J_{n}(v_{n})=\inf\{J_{n}(w):w\in u*\varphi_{\frac{1}{n}}+W^{1,p}_{0}(B)\}.
\]

{\em Step 2} (Determining the parameters)\\
For any choice of $\eps_{n}, \mu_n$, the integrand
\[
F_n(z):=F*\varphi_{\eps_{n}}(z) +\frac{\mu_{n}}{2}\, |z|^{2}
\]
is  $\mu_{n}$-uniformly convex, hence \cite[Theorem 4.1]{BB} ensures the existence of constants $A_{n}$ (depending only on $B$ and $\|f_{n}\|_{\infty}$, as well as on the regularity of $u*\varphi_{\frac{1}{n}}$, but not on $\eps_{n}$, $\mu_{n}$), such that 
\beq
\label{defln}
{\rm Lip}(v_{n})\le \frac{A_{n}}{\mu_{n}}=:L_n.
\eeq
Without loss of generality, we can assume $ L_n \ge 1 $. We first choose $\mu_{n}\downarrow 0$ so that 
\beq
\label{condmun}
\lim_{n}\mu_{n}\int_{B}|Du*\varphi_{\frac{1}{n}}|^{2}\, dx=0,\qquad \lim_{n}\mu_{n}^{p-1}\, A_{n}^{2-p} =0,
\eeq
and observe that $L_{n}$ is independent of $\eps_{n}$. Then we choose $\eps_{n}$:  define the numbers
\[
M_{n}=1+\sup_{B} |Du*\varphi_{\frac{1}{n}}| + L_n.
\]
Since $\varphi_\eps * F\to F$ in $C^1_{\rm loc}(\R^N)$ as $\eps\downarrow 0$, we can pick $(\eps_{n})\subseteq (0, 1)$, $\eps_{n}\downarrow 0$, so that 
\beq
\label{epsn}
\|F*\varphi_{\eps_n}-F\|_{C^1(B_{M_n})}\le \frac{1}{n}.
\eeq
Clearly, it still holds
\beq
\label{fn}
F_{n}\to F\qquad \text{in $C^{1}_{{\rm loc}}(\R^{N})$}.
\eeq

{\em Step 3} (The limsup inequality)\\
Testing the minimality of $v_{n}$ against the admissible function $u*\varphi_{\frac{1}{n}}$ gives 
\[
J_{n}(v_{n})\le  J_{n}(u*\varphi_{\frac{1}{n}}).
\]
Owing to $u*\varphi_{\frac{1}{n}}\to u$ in $W^{1,p}(B)$ and \eqref{condmun}, one has
\beq
\label{varlimsup}
\varlimsup_{n} J_{n}(u*\varphi_{\frac{1}{n}})=\int_{B}f\, u\, dx+\varlimsup_{n} \int_{B} F*\varphi_{\eps_{n}}(Du*\varphi_{\frac{1}{n}})\, dx.
\eeq
To estimate the last integral, we use \eqref{epsn} to get
\[
 \int_{B} F*\varphi_{\eps_{n}}(Du*\varphi_{\frac{1}{n}})\, dx\le \frac{|B|}{n}+\int_{B}F(Du*\varphi_{\frac{1}{n}})\, dx.
 \]
The vector-valued Jensen inequality then leads to
\[
\begin{split}
 \varlimsup_{n} \int_{B} F*\varphi_{\eps_{n}}(Du*\varphi_{\frac{1}{n}})\, dx&\le \varlimsup_{n} \int_{B} F(Du*\varphi_{\frac{1}{n}})\, dx\\
 &\le \varlimsup_{n} \int_{(1+1/n)B} \varphi_{\frac{1}{n}}*F(Du)\, dx\\
 &=\int_{B}F(Du)\, dx.
 \end{split}
 \]
Inserting the latter into \eqref{varlimsup} we conclude
 \beq
 \label{limsup}
 \varlimsup_{n} J_{n}(v_{n})\le J(u).
 \eeq

 {\em Step 4} (Convergence of $(v_{n})$ to $u$)\\
 From \eqref{limsup} we have that $(J_{n}(v_{n}))$ is bounded. By Jensen's inequality $F\le F *\varphi_{\eps_{n}}$ so that, through \eqref{fdf}, for some constant $C=C(N, F, B)>0$ we have
\[
 \begin{split}
 J_{n}&(v_{n})\ge \int_{B} F(Dv_{n})\, dx+\int_{B}f_{n}\, v_{n}\, dx\\
 &\ge  \frac{\|Dv_{n}\|_{L^{p}(B)}^{p}}{C} - C-\|f_{n}\|_{W^{-1, p'}(B)}\|D(v_{n}-u*\varphi_{\frac{1}{n}})\|_{L^{p}(B)}+\int_{B} f_{n}\, u*\varphi_{\frac{1}{n}}\, dx\\
 & \ge \frac{\|Dv_{n}\|_{L^{p}(B)}^{p}}{C} - C-\|f_{n}\|_{W^{-1, p'}(B)}\big(\|Dv_{n}\|_{L^{p}(B)}+\|D u*\varphi_{\frac{1}{n}}\|_{L^{ p}(B)}\big)+\int_{B} f_{n}\, u*\varphi_{\frac{1}{n}}\, dx.
  \end{split}
 \]
 Since $f_{n}\to f$ in $W^{-1, p'}(B)$ and $u*\varphi_{{1/n}}\to u$ in $W^{1,p}(B)$, by \eqref{limsup} we deduce that $(Dv_{n})$ is bounded in $L^{p}(B)$. Moreover, by Poincar\'e's inequality, 
 \[
 \begin{split}
 \|v_{n}\|_{L^{p}(B)}&\le \|v_{n}-u*\varphi_{\frac{1}{n}}\|_{L^{p}(B)}+\|u*\varphi_{\frac{1}{n}}\|_{L^{p}(B)}\\
 &\le C\, \big(\|D(v_{n}-u*\varphi_{\frac{1}{n}})\|_{L^{p}(B)}+\|u*\varphi_{\frac{1}{n}}\|_{L^{p}(B)}\big)\\
 &\le C\, \big(\|Dv_{n}\|_{L^{p}(B)}+\|u*\varphi_{\frac{1}{n}}\|_{W^{1,p}(B)}\big),
 \end{split}
 \]
 so that $(v_{n})$ is bounded in $L^{p}(B)$ as well. Therefore the sequence $(v_{n})$ is bounded in $W^{1,p}(B)$, and hence possesses a (not relabeled) subsequence weakly convergent to some $v\in W^{1,p}(B)$; actually, it is readily checked that $v\in u+W^{1,p}_{0}(B)$. We claim that
\[
\lim_n \mu_{n}\int_{B}|Dv_{n}|^{2}\, dx =0.
\]
Indeed, this is obvious by H\"older's inequality when  $p\ge 2$, while if $p<2$ we use \eqref{defln} and \eqref{condmun} to infer
\[
\mu_{n}\int_{B}|Dv_{n}|^{2}\, dx\le\mu_{n}\, L_{n}^{2-p}\int_{B} |Dv_{n}|^{p} \, dx\le  \mu_{n}^{p-1}\, A_{n}^{2-p}\int_{B}|Dv_{n}|^{p}\, dx\quad \to \quad 0,
\]
where we used the boundedness of $(Dv_{n})$ in $L^{p}(B)$. Thus
\beq
\label{mnvn}
\lim_n\int_{B}\frac{\mu_{n}}{2}\,  |Dv_n|^{2}+f_{n}\, v_{n}\, dx=\int_{B}f\, v\, dx.
\eeq
The functional
\[
w\mapsto \int_{B}F(Dw)\, dx
\]
is weakly lower semicontinuous in $W^{1,p}(B)$, whence (again by the Jensen inequality)
\beq
\label{qpl}
\begin{split}
J(v)&\le \varliminf_{n} \Bigg[ \int_{B} F(Dv_{n})\, dx+\int_B f\, v_n\, dx \Bigg] \\
&\le\varliminf_n \Bigg[ \int_{B} F*\varphi_{\eps_n}(Dv_{n})\, dx + \frac{\mu_n}{2} \int_{B} |Dv_n|^{2}+f_{n}\, v_{n}\, dx \Bigg] =\varliminf_{n} J_{n}(v_{n}).
\end{split}
\eeq
 Coupling the latter with \eqref{limsup} we get $ J(v)\le J(u)$, implying $v=u$ by the strict convexity of $J$. In particular we obtain, up to subsequences,
 \beq
 \label{wvn}
 Dv_{n}\rightharpoonup Du \qquad \text{in $L^{p}(B)$},
 \eeq
 and from \eqref{qpl}, \eqref{limsup} we infer $J_n(v_n)\to J(u)$. Subtracting  \eqref{mnvn} we get
\[
 \int_B F*\varphi_{\eps_n}(Dv_n)\to \int_B F(Du)
\]
which, thanks to \eqref{epsn}, implies
  \beq
 \label{fn3}
  \int_B F(Dv_n)\to \int_B F(Du).
  \eeq
 
 {\em Step 5} (Uniform Sobolev bound on $DF(Dv_{n})$)\\ 
By Proposition \ref{proqu}, point (iii), and the beginning of the proof of Theorem \ref{Mth}, $F_{n}$ fulfills \eqref{Kell} with the same constant $K$; since $F_{n}\in C^{3}(\R^N)$, \eqref{Kell} actually holds everywhere. Moreover, standard regularity theory ensures that $v_{n}\in C^{2}(B)$, so we can apply Theorem \ref{Mth}, and in particular \eqref{w12est}, to obtain, for any  $\theta\in (0, 2]$ and $r=1/2, 1$,
\beq
\label{h}
\|DF_{n}(Dv_{n})\|_{W^{1,2}(\frac r 2 B)}\le C\, \big(\|f_{n}\|_{L^{2}(rB)}+\|DF_{n}(Dv_{n})\|_{L^{\theta}(rB)}\big).
\eeq
The first term on the right is clearly bounded by a multiple of $\|f\|_{L^{2}(B)}$. For the second one, we let $p, q$ be given in \eqref{fdf} and  choose $\bar\theta=\min\{p/(q-1), 1\}$. By \eqref{fdf} we get
\beq
\label{growth}
|DF(z)|^{\bar\theta}\le C\, \big(|z|^{\bar\theta\, (q-1)}+1\big)\le C\, (|z|^p+1)^{\bar\theta\, (q-1)/p}\le C \big(F(z)+1\big)^{\bar\theta\, (q-1)/p}.
\eeq
Using \eqref{epsn} and \eqref{growth} we obtain
\beq
\label{vvvv}
\begin{split}
\int_{B} &|DF_{n}(Dv_{n})|^{\bar\theta}\, dx \le \int_{B}|DF(Dv_{n})|^{\bar\theta}\, dx+n^{-\bar\theta}\, |B| +  \mu_n^{\bar\theta} \int_B |Dv_n|^{\bar\theta}\, dx\\
&\le C\, \int_B\big(F(Dv_n)+1\big)^{\bar\theta\, (q-1)/p}\, dx+n^{-\bar\theta}\, |B| + \mu_n^{\bar\theta}\, \|Dv_n\|_{L^{ p}(B)}^{\bar\theta}|B|^{1-\bar\theta/p}\\
&\le C\, |B|^{1-\bar\theta\, (q-1)/p}\Big(\int_B \big(F(Dv_n)+1\big)\, dx\Big)^{\bar\theta\, (q-1)/p}+n^{-\bar\theta}\, |B| + \mu_n^{\bar\theta}\, \|Dv_n\|_{L^{ p}(B)}^{\bar\theta}|B|^{1-\bar\theta/p}.
\end{split}
\eeq
The first integral is bounded by   \eqref{fn3} and the remaining terms vanish when $ n \to \infty $, so
\beq
\label{fert}
\varlimsup_n\|DF_n(Dv_n)\|_{L^{\bar\theta}(B)}\le C\,\Big(\int_B \big(F(Du)+1\big)\, dx\Big)^{(q-1)/p}.
\eeq
Thanks to \eqref{h} for $r=1$, \eqref{fert} implies the Sobolev bound
\beq
\label{gal}
\varlimsup_{n} \|D F_{n}(Dv_{n})\|_{W^{1, 2}(\frac{1}{2}B)}\le C.
\eeq

{\em Step 6} (Passage to the limit)\\
Let $B'=\frac{1}{2}B$ and
\[
V_{n}=DF_{n}(Dv_{n}).
\]
Thanks to \eqref{gal}, $(V_{n})$ is bounded in $W^{1,2}(B')$, hence we can pick a subsequence satisfying 
\beq
\label{conv}
V_{n}\to V\qquad \text{weakly in $W^{1,2}(B')$, strongly in $L^{2}(B')$, and pointwise a.\,e.\, in $B'$,}
\eeq
for a suitable $V\in W^{1,2}(B')$.

Each $F_{n}$ is strictly convex and superlinear by construction, hence $DF_{n}$ is a homeomorphism of $\R^{N}$. Moreover, by \eqref{fn} we know that $DF_{n}\to DF$ locally uniformly. According to a theorem by Arens (see \cite{D} for a modern exposition), this implies that $DF_{n}^{-1}\to DF^{-1}$ locally uniformly. Since $V_{n}\to V$ pointwise a.\,e., we infer that 
\[
Dv_{n}=DF_{n}^{-1}(V_{n})\to DF^{-1}(V)\qquad \text{pointwise a.\,e.},
\]
which, coupled with \eqref{wvn}, allows the identification $Du=DF^{-1}(V)$. Therefore, $V_{n}\to DF(Du)$ in $B'$ in all the senses prescribed in \eqref{conv}. 

By considering  $4 B$ instead of $B$, estimate \eqref{mainest2} follows from \eqref{h} and \eqref{fert}, as long as $4 B\subseteq \Omega$.
Let $B''=\frac{1}{4} B$. By Lebesgue's Dominated Convergence theorem $\|V_{n}\|_{L^{\theta}(B')}\to \|V\|_{L^{\theta}(B')}$ hence, exploiting also the lower semi-continuity of the $W^{1, 2}(B'')$ norm, we can pass to the limit in \eqref{h} with $r=1/2$. Again considering $4 B$ instead of $B$ yields  \eqref{mainest}. 
Finally, the validity of \eqref{EL} can be checked only on balls $B$ such that $4B\subseteq \Omega$, by a standard partition of unity argument. If $B$ is such a ball, we can pass to the limit in the Euler Lagrange equations for the approximating problems constructed as before in $4B$, and since $DF_n(Dv_n)\to DF(Du)$, $f_n\to f$ strongly in $L^2(B)$, we get \eqref{EL}.
 \end{proof}

\begin{remark}\label{rempq}
The previous theorem has an immediate consequence. The class of q.\,u.\,c.\,functional is a (proper) subclass of the so-called {\em functionals with $(p, q)$-growth}, i.\,e.\, those obeying \eqref{pqg}, \eqref{pqg2}. For example, the integrand
\[
F(z)=|z_1|^p+|z_2|^q, \qquad z=(z_1, z_2)\in \R^2, \qquad p, q>1
\]
is of $(p, q)$-growth but not q.\,u.\,c.\,unless $p=q$. Given a local minimizer of a convex functional of {\em general} $(p, q)$-growth, the a-priori regularity on $Du$ is just $Du\in L^p_{\rm loc}(\Omega)$, and the first step towards higher regularity is showing that actually $Du\in L^q_{\rm loc}(\Omega)$. For  $2\le p\le N$, this is to be expected  only when $q<p\, (N+2)/N$; see \cite{Gi, ELM1}. For the sub-class of  q.\,u.\,c.\,integrands, from $DF(Du)\in W^{1,2}_{\rm loc}(\Omega)$, we infer by Sobolev's embedding that $DF(Du)\in L^{2^*}_{\rm loc}(\Omega)$.  Since $|DF(z)|\succsim |z|^{p-1}$, it holds $Du\in L^{2^*(p-1)}_{\rm loc}(\Omega)$, hence for such class of integrands we obtain the  condition  $q\le 2^*\, (p-1)$, which gives a larger range if $p\ge 2$. This range may not be optimal in the q.\,u.\,c.\,class, but on one hand it shows the advantages of considering the stress field instead of \eqref{defvp}, while on the other hand it holds for any $f\in L^2_{\rm loc}(\Omega)\cap W^{-1, p'}(\Omega)$. It is quite possible, in light of the results of \cite{BM}, that, for $f$ having a sufficiently high degree of summability, minimizers for q.\,u.\,c.\,integrands are automatically Lipschitz continuous, regardless of the largeness of the ratio $q/p$, a fact that, if true, would bypass completely the higher integrability issue for the gradient in the q.\,u.\,c.\,class. This is actually the case for functionals with Uhlenbeck structure, see \cite{CM2}.
\end{remark}

\subsection{Examples}\label{examples}

\begin{example}[On the assumption $F\in W^{2,1}$]\label{excantor}
In this example we show that, in order to obtain Sobolev regularity of $DF(Du)$, it is not sufficient to require that  condition \eqref{Kell} holds at almost every point, but that Sobolev regularity of $DF$ is a necessary assumption.

Let $N=2$. For any ball $B\Subset \{(x, y)\in \R^{2}:x>0\}$, consider the smooth function
\[
u(x, y)=\arctan(y/x).
\]
We claim that, for any (not necessarily convex) $C^{2}$  function $F:\R\to \R$, $u$ solves
\beq	
\label{claimf}
\Div DF(Du)=0\qquad \text{in $B$},
\eeq
where here and in what follows we make the identification $F(z)=F(|z|)$. Letting $z=(x, y)$, $z^{\bot}=(-y, x)$, it holds
\[
Du(z)=\frac{z^{\bot}}{|z|^{2}}, \qquad D^{2}u(z)=\frac{1}{|z|^{4}}
\begin{pmatrix}
2\, x\, y&y^{2}-x^{2}\\
y^{2}-x^{2}&-2\, x\, y
\end{pmatrix},
\]
while
\[
DF(w)=F'(|w|)\frac{w}{|w|},\qquad D^{2}F(w)=F''(|w|)\frac{w}{|w|}\otimes \frac{w}{|w|}+\frac{F'(|w|)}{|w|}\left(I-\frac{w}{|w|}\otimes \frac{w}{|w|}\right).
\]
An elementary computation then yields
\[
D^{2}F(Du)\, D^{2}u(z)=\frac{1}{|z|^{4}}\begin{pmatrix}
x\, y\, \big(F''(\frac{1}{|z|})+F'(\frac{1}{|z|})\, |z|\big)&F''(\frac{1}{|z|})\, y^{2}-F'(\frac{1}{|z|})\, |z|\, x^{2}\\[10pt]
F'(\frac{1}{|z|})\, |z|\, y^{2}- F''(\frac{1}{|z|})\, x^{2}&- x\, y\, \big(F''(\frac{1}{|z|})+F'(\frac{1}{|z|})\, |z|\big)
\end{pmatrix}
\]
which has zero trace, proving the claim.

Now, let $h:[0, 1]\to[0, 1]$ denote the Cantor function. By abuse of notation, we still denote by $h$ its extension to the whole $\R$ defined as
\[
h(t)=k+h(t-k)\quad \text{if $k\le t<k+1$, $k\in \Z$},
\]
and we also denote by $\cal C$ the periodic extension of the Cantor set to the whole $\R$.
Consider
\[
F(|w|):=\frac{|w|^{2}}{2}+H(|w|),\qquad H(t):=\int_{0}^{t}h(\tau)\, d\tau,
\]
which is a strictly convex $C^{1}$ function with quadratic growth. Clearly, $F$ can be approximated in $C^{1}$ by a sequence $\{F_{n}\}$ of smooth radial functions, so that we can pass to the limit into the corresponding weak formulations of \eqref{claimf} to obtain that $u$ solves 
\[
\Div DF(Du)=0\qquad \text{weakly in $B$}.
\]
However, 
\[
DF(Du(z))=\frac{z^{\bot}}{|z|^{2}} +h(|z|^{-1})\, \frac{z^{\bot}}{|z|}
\]
is not even absolutely continuous in $B$, since its distributional derivative has a Cantor part concentrated on $\{z:1/|z|\in {\cal C}\}$, which has zero measure. 

Since $h'(t)=0$ in the classical sense for a.\,e.\, $t\in \R$, it is readily verified that $F$ obeys  \eqref{Kell} with $K=1$ at every point of 
\[
\R^{N}\setminus \cal C_{{\rm rad}}, \qquad \cal C_{{\rm rad}}=\{z\in \R^{N}:|z|\in \cal C\},
\]
thus almost everywhere. 
Notice that $DF$ is of bounded variation  but does not belong to $W^{1,1}_{\rm loc}(\R^N)$, since its derivative has a Cantor part concentrated on ${\cal C}_{\rm rad}$.
\end{example}

\begin{example}[Uhlenbeck structure]\label{exuhl}
For divergence form equations having the Uhlenbeck structure
\beq
\label{Uhlenbeck}
\Div(a(|Du|)\, Du)=f
\eeq
we recover the local regularity result of \cite[Theorem 2.1]{CM}, under the additional assumption $f\in W^{-1, p'}(\Omega)$ (see the second point in Remark \ref{rem0} in this respect).  Indeed we can define $F$  by
\[
F(z)=\int_0^{|z|} t\, a(t)\, dt
\]
to get $DF(z)=a(|z|)\, z$. If $a\in C^{1}(0, +\infty)$, it holds
\[
D^{2}F(z)=a(|z|)\, I+|z|\, a'(|z|)\, \frac{z}{|z|}\otimes \frac{z}{|z|},
\]
possessing the eigenvector $z/|z|$ with eigenvalue $a(|z|)+|z|\, a'(|z|)$, while its orthogonal eigenspace is relative to the unique eigenvalue $a(|z|)$. The equation is elliptic if and only if both eigenvalues are non-negative, and in order to bound the ratio between them we look at 
\[
K=\sup_{t>0}\max\left\{\frac{a(t)}{a(t)+t\, a'(t)}, \frac{a(t)+t\, a'(t)}{a(t)}\right\}.
\]
It is readily checked that if
\[
i_{a}=\inf_{t>0}\frac{t\, a'(t)}{a(t)},\qquad s_{a}=\sup_{t>0}\frac{t\, a'(t)}{a(t)},
\]
then 
\[
K=\max\left\{\frac{1}{1+i_{a}}, 1+s_{a}\right\},
\]
so that the q.\,u.\,convexity condition \eqref{Kell} is equivalent to the common requirement
\beq
\label{indexes}
-1<i_{a}\le s_{a}<\infty,
\eeq
which is the one used, e.\,g., in \cite{CM}. Notice that when $a(z)=|z|^{p-2}$, which corresponds to the $p$-Poisson equation, we have $i_{a}=s_{a}=p-2$, so that $K=\max\{p-1, \frac{1}{p-1}\}$.

When $f$ is sufficiently smooth (say, bounded), equation \eqref{Uhlenbeck} can be seen as the Euler-Lagrange equation for the functional \eqref{defJ}, which is strictly convex. Moreover, it holds (see \cite[Proposition 2.15]{CM2}) 
\[
F(z)\ge a(1)\frac{|z|^{2+i_{a}}}{2+i_{a}}, \qquad |z|\ge 1,
\]
so that $J$ is also coercive on $W^{1, 2+i_{a}}(\Omega)$ when supplemented with reasonable boundary conditions. Since $2+i_{a}>1$ by \eqref{indexes}, the variational treatment of \eqref{Uhlenbeck} in standard Sobolev spaces is thus justified if one is not looking for optimal rearrangement invariant estimates.
\end{example}

\begin{example}[Anisotropic examples]\label{Cir}
In \cite{CFR, ACF, CFV1, CFV2} anisotropic equations whose principal part arises as the Euler-Lagrange equation of 
\[
\int_\Omega G(H(Du))\, dx
\]
are considered, where $H\in C^2(\R^N\setminus\{0\}, \R_+)$ is a convex, positively $1$-homogeneous function and $G\in C^2(\R_+, \R_+)$ is an increasing, strictly convex function of $p$-growth. Clearly, $H$ is fully determined by the unit ``ball" 
\[
B_H=\{z\in \R^N: H(z)<1\}\ni 0,
\]
and any open, bounded, convex $B$ with $0\in B$ will uniquely determine such an $H$ through its Minkowski functional. Notice that in general $H$ may not even be a norm, due to the possible lack of symmetry of $B_H$. 

The more general ellipticity assumption in the cited works reads as follows: $H$ is said to be {\em uniformly convex} if the principal curvatures of $\partial B_H$ are bounded from below by a positive constant. From \cite[Appendix A]{CFV1}, the uniform ellipticity of $H$ amounts to 
\[
\big(H(z)\, D^2H(z) \, \eta, \eta\big)\ge \delta\,  |\eta|^2\qquad \forall z\in \R^N\setminus \{0\}, \ \eta\in DH(z)^\bot,
\]
for some $\delta>0$ (which is actually equivalent to the same type of inequality for $\eta\in z^\bot$).
Under this assumption various results can be proved, and in particular the Sobolev regularity of the associated stress field is treated in \cite{CFR, ACF} as a stepping-stone to more general results. 
Here we show that anisotropic functionals of this kind fall within our general framework of q.\,u.\,c.\,functionals.

To this end, set $F(z)=G(H(z))$ so that 
\[
D^2F(z)=G''(H(z))\, DH(z)\otimes DH(z)+\frac{G'(H(z))}{H(z)}\, H(z)\, D^2 H(z)
\]
for $z\ne 0$, and notice that both
\[
DH(z)\otimes DH(z) \quad \text{and}\quad H(z)\, D^2 H(z)
\]
are $0$-homogeneous.
Inspecting the proof of \cite[Appendix A]{CFV1}, we see that for any $z, \xi\in \R^N\setminus\{0\}$ it either holds
\[
 \big(DH(z)\otimes DH(z)\, \xi, \xi\big)\ge \lambda_1\, |\xi|^2
\]
or
\[
 \big(H(z)\, D^2H(z)\, \xi, \xi\big)\ge \lambda_2\, |\xi|^2
\]
(with $\lambda_i=\lambda_i(H)>0$), while altogether
\[
\big(H(z)\, D^2H(z)\, \xi, \xi\big)\le \Lambda\, |\xi|^2, \qquad   \big(DH(z)\otimes DH(z)\, \xi, \xi\big)\le \Lambda\, |\xi|^2
\]
for some $\Lambda=\Lambda(H)$. It follows that the minimum eigenvalue of $D^2F(z)$ is bounded from below by 
\[
\min\big\{\lambda_1 \, G''(H(z)), \lambda_2\, G'(H(z))/H(z)\big\},
\]
while its maximum eigenvalue is bounded from above by 
\[
\Lambda \big( G''(H(z))+G'(H(z))/H(z)\big).
\]
Therefore $F$ is $K$-q.\,u.\,c.\,for 
\[
K=\frac{\Lambda}{\min\{\lambda_1, \lambda_2\}}\sup_{t\in \R_+}\max\left\{1+\frac{G'(t)}{G''(t)\, t}, 1+\frac{G''(t)\, t}{G'(t)}\right\}.
\]
The finiteness of the latter is compatible with the standard Uhlenbeck example: for $a(t)=G'(t)/t$, a straightforward calculation shows
that 
\[
\sup_{t\in \R_+}\max\left\{\frac{G'(t)}{G''(t)\, t}, \frac{G''(t)\, t}{G'(t)}\right\}<+\infty\quad \Leftrightarrow \quad -1<i_a\le s_a<+\infty.
\]

\end{example}

\begin{example}[Non-standard anisotropic growth]\label{exani}
It is straightforward to check that if $F_{1}$ and $F_{2}$ are q.\,u.\,c.\, $C^{1}$ functions fulfilling \eqref{Kell} with constant $K_{1}$ and $K_{2}$ respectively, then $F_{1}+F_{2}$ is $\max\{K_{1},K_{2}\}$-q.\,u.\,c.. This observation allows to consider anisotropic Euler-Lagrange equations arising from integrands of the form 
\[
 F(z)=\sum_{i=1}^{M}|A_{i}(z-z_{i})|^{p_{i}}, \qquad z_{i}\in \R^{N},\quad A_{i}\ge I, \quad p_{i}>1 \quad \text{for $i=1, \dots, M$}
 \]
As a  more common example of anisotropic functionals, consider the weak solution of the Dirichlet problem
 \[
\begin{cases}
\Delta_{p} u+D_{1}(|D_{1}u|^{q-2}\, D_{1}u)=f&\text{in $B$}\\
u=0&\text{on $\partial B$}
\end{cases}
 \]
 for a ball $B$, $f\in L^{\infty}(B)$ and $q>p>2$. This corresponds to the unique minimizer $u\in W^{1, p }_{0}(B)$ of 
 \[
 J(w)=\int_{B}\frac{1}{p}|Dw|^{p}+\frac{1}{q}|D_{1}w|^{q}+f\, w\, dx,
 \]
 and it is globally Lipschitz continuous, since 
 \[
 F(z)=\frac{1}{p}|z|^{p}+\frac{1}{q}|z_{1}|^{q}
 \]
 is uniformly $p$-convex outside $B_{1}$ (see \cite{BB}), i.\,e.,
 \[
 (D^{2}F(z)\, \xi, \xi)\ge c_{p}(1+|z|^{2})^{\frac{p-2}{2}}|\xi|^{2},\qquad |z|\ge 1.
 \]
 Moreover, 
 \[
 D^{2}F(z)=|z|^{p-2}\Big(I+(p-2)\, \frac{z}{|z|}\otimes\frac{z}{|z|} \Big)+(q-1)|z_1|^{q-2} \, e_1 \otimes e_1,
 \]
 so that the minimum and maximum eigenvalues of $D^{2}F(z)$ fulfill
 \[
 \lambda_{{\rm min}}(z)\ge |z|^{p-2},\qquad \lambda_{{\rm max}}(z)\le (p-1)\, |z|^{p-2}+(q-1)\, |z|^{q-2}.
 \]
 Notice that the integrand $F$ is {\em not} q.\,u.\,c.\,globally on $\R^N$, but on the range of $Du$ we have $|z|^{q-2}\le C\, |z|^{p-2}$ for $C$ depending on ${\rm Lip}\, u$, leading to
 \[
 \lambda_{{\rm max}}(z)/\lambda_{{\rm min}}(z)\le p-1+C\, (q-1).
 \]
 Hence Theorem \ref{minloc} applies thanks to the local nature of \eqref{Kell} described in Lemma \ref{locality}.
 Notice the role of the assumption $f\in L^{\infty}(B)$  (which could be weakened, but not down to $L^{2}(B)$, see \cite{BM}) and of the smooth boundary condition $u\in W^{1,p}_0(B)$: they provide the Lipschitz regularity of $u$, which in turn allows to employ Lemma \ref{locality} and to consider $u$ as a minimizer of a functional with  q.\,u.\,c.\,integrand on the whole $\R^N$.
\end{example}

\section{Applications}\label{5}

\subsection{Cordes-type conditions}

We start with a generalization of \eqref{asl}. In this section, given a matrix field $M:\Omega\to \R^{N}\otimes \R^{N}$, its $L^{m}$-norm will be computed with respect to the Frobenius norm of $M=(m_{ij})$, i.\,e.,
\[
\|M\|_{m}=\left(\int |M(x)|_{2}^{m}\, dx\right)^{1/m},\qquad |M(x)|_{2}^{2}=\sum_{i, j=1}^{N}|m_{ij}(x)|^{2}.
\]

The proofs of this section make use of some tools from Harmonic Analysis; for an introductory exposition on this topic, the reader can consult \cite{Duo}.

\begin{lemma}
Let $V\in C^{1}_{c}(\R^{N}; \R^{N})$ and, for $m>1$, set $\hat{m}=\max\{m, m/(m-1)\} \ge 2 $. Then
\beq
\label{lp}
\|DV\|_{m}\le N^{2}\, (\hat{m}-1)\,\Big( \|\Div V\|_{m}+\|\curl V\|_{m}\Big).
\eeq
\end{lemma}

\begin{proof}
Let $R_{j}$ be the Riesz transform, defined as the Fourier multiplier with symbol $-i\, \xi_{j}/|\xi|$ (cf. \cite[p. 76]{Duo}). Then it holds
\beq
\label{dvhk}
D_{h}V_{k}=-R_{h}\, R_{k}\, \Div V-\sum_{j=1}^{N}R_{h}\, R_{j}\, \curl_{kj} V,
\eeq
which follows, taking the Fourier transform (denoted by $g\mapsto \hat{g}$), from the identity
\[
i\, \xi_{h}\, \hat V_{k}=\frac{\xi_{h}\, \xi_{k}}{|\xi|^{2}}\sum_{j=1}^{N} i\, \xi_{j}\, \hat V_{j}+\sum_{j=1}^{N}\frac{\xi_{h}\, \xi_{j}}{|\xi|^{2}}\Big(i\, \xi_{j}\, \hat V_{k}-i\, \xi_{k}\, \hat V_{j}\Big).
\]
The second order Riesz transform has $L^{m}-L^{m}$ norm $\hat{m}-1$ on diagonal terms and $(\hat{m}-1)/2$ on off-diagonal terms, i.\,e.,
\[
\begin{split}
\|R_{h}\, R_{k} \,g\|_{m}&\le \frac{\hat{m}-1}{2}\, \|g\|_{m}, \qquad h\ne k,\\
\|R_{h}^{2}\, g\|_{m}&\le (\hat{m}-1)\,\|g\|_{m}
\end{split}
\]
(see \cite[Theorem 2.4]{BMH}). Thus from \eqref{dvhk} we get 
\[
\|D_{h}V_{k}\|_{m}\le (\hat{m}-1)\left(\|\Div V\|_{m}+\sum_{j=1}^{N}\|\curl_{kj} V\|_{m}\right).
\]
We sum over $k=1, \dots, N$ and use the H\"older inequality to get
\[
\sum_{k, j=1}^{N}\|\curl_{kj} V\|_{m}\le N^{2\, (1-\frac{1}{m})}\left(\int \sum_{k, j=1}^N |\curl_{kj} V|^m\right)^{1/m}. 
\]
Since
\[
\sum_{k, j=1}^N |\curl_{kj} V|^m\le N^{2-m}\, \left( \sum_{k, j=1}^N |\curl_{kj} V|^2\right)^{m/2}
\]
we obtain
\[
\sum_{k, j=1}^{N}\|\curl_{kj} V\|_{m}\le N\, \|\curl V\|_m.
\]
Summing also over $h=1, \dots, N$ we finally get
\[
\|DV\|_{m}\le \sum_{h, k=1}^{N}\|D_{h}V_{k}\|_{m}\le (\hat{m}-1)\, N^{2} \Big(\|\Div V\|_{m}+\|\curl V\|_{m}\Big).
\]
\end{proof}

\begin{remark}
\label{nonoptimal}
Estimate \eqref{lp} is very rough in its dependence on $N$. It is a common feature of $L^{m}$-bounds on Riesz transform that they do not depend on the dimension of the euclidean space. Up to our knowledge, optimal $L^m$ estimates for the  operator $D \curl^{-1}$ (let alone for the resolvent operator of the  $\Div-\curl$ system) are not known. It is also not optimal as $m\to 2$, compare it with \eqref{asl} in the case $m=2$.
\end{remark}

\begin{theorem}\label{cord}
Let $F$ obey the assumptions of Theorem \ref{minloc}, in particular \eqref{Kell} with a constant $K\ge 1$ and $p, q$ given in Proposition \ref{proqu}. Let furthermore  $m>1$ and $f\in L^{m}(\Omega)\cap W^{-1, p'}(\Omega)$. Then any  local minimizer $u\in W^{1, p}_{\rm loc}(\Omega)$ for $J$ given in \eqref{defj2} is such that $DF(Du)\in W^{1, m}_{\rm loc}(\Omega)$ and fulfills the estimates
\beq
\label{v0bis}
 \|DF(Du)\|_{L^{m}(B)}\le C\, \Big(1+\|f\|_{L^m(2B)}+\|F(Du)\|_{L^1(2B)}^{(q-1)/p}\Big)
 \eeq
\beq
\label{v0}
\|DF(Du)\|_{W^{1, m}(B)}\le C\, \Big(\|f\|_{L^{m}(2B)}+ \|DF(Du)\|_{L^{m}(2B)}\Big)
\eeq
in each of the following cases:
\begin{enumerate}
\item
$K\le K_{0}$, with $K_{0}>1$ depending on $N$ and $m$, with $C=C(N, m, B)$.
\item
$|m-2|\le \delta_{0}$, for $\delta_{0}>0$ depending on $N$ and $K$, with $C=C(N, K, B)$.
\end{enumerate}
\end{theorem}

\begin{proof}
Given $B$ such that $4B\Subset \Omega$, we follow the first four steps of the proof of Theorem \ref{minloc}, to find $u_{n}\in C^{\infty}(B)$, $f_{n}\in C^{\infty}(B)$, $F_{n}\in C^{\infty}(\R^{N})$ and $v_{n}\in C^{2}(B)$ such that 
\begin{enumroman}
\item
$u_{n}\to u$ in $W^{1, p}(B)$, $f_{n}\to f$ in $L^{m}(B)\cap W^{-1, p'}(B)$;
\item 
$F_{n}$ is $K$-q.\,u.\,c.;
\item
$v_{n}\rightharpoonup u$ in $W^{1, p}(B)$, $v_{n}\to u$ in $L^{p}(B)$, and $v_{n}$ solves
\[
\Div(DF_{n}(Dv_{n}))=f_{n};
\]
\item
$\|F_{n}(Dv_{n})\|_{L^{1}(B)}\to \|F(Du)\|_{L^{1}(B)}$.
\end{enumroman} 
We then proceed as in Theorem \ref{Mth}, in order to find a uniform bound for $DF_{n}(Dv_{n})$ in $W^{1, m}(B')$, $B'=\frac{1}{2}B$. To simplify the notation, we omit for the moment the dependence of $v$, $F$, and $f$ on $n$.

Let $V=DF(Dv)$ and observe that the assumptions of  Lemma \ref{basiclemma} hold true for the matrix $X=DV=D^{2}F(Dv)\, D^{2}v$; hence, pointwise,
\beq
\label{vv}
|\curl V|_{2}\le \sqrt{2}\, \e(K)\, |DV|_{2}, \qquad \e(K)=1-\frac{1}{K}.
\eeq
Suppose $B=B_{2R}$ and, for any $R\le r< s\le 2\, R$,  fix $\varphi\in C^{\infty}_{c}(B_{s}; [0, 1])$ as in \eqref{propvarphi}. 

We split the proof in two cases, according to the situations under consideration in the two statements. For the first assertion, we apply \eqref{lp} to the field $W:=\varphi\, V$ to get
\[
\begin{split}
\|\varphi\, DV\|_{m}&\le \|DW\|_{m}+\|D\varphi\otimes V\|_{m}\\
&\le N^{2}\, (\hat{m}-1)\big(\|\Div W\|_{m}+ \|\curl W\|_{m}\big) + \frac{C}{s-r}\|V\|_{L^{m}(B_{s}\setminus B_{r})}
\end{split}
\]
for any $m>1$. From
\[
\Div W=\varphi\, f +(V, D\varphi),\qquad \curl W=\varphi\, \curl V+V\land D\varphi
\]
we thus infer
\beq
\label{v}
\|\varphi\, DV\|_{m}\le N^{2}\, (\hat{m}-1)\big( \|\varphi\, \curl V\|_{m}+\|\varphi\, f\|_{m}\big)+\frac{C}{s-r}\|V\|_{L^{m}(B_{s}\setminus B_{r})}.
\eeq
We then let $K_{0}=K_{0}(N, m)>1$ satisfying
\beq
\label{condkzero}
\sqrt{2}\, N^{2}\, (\hat{m}-1)\, \e(K_{0})< 1
\eeq
so that, if $K\le K_{0}$,  the curl term in \eqref{v} can be reabsorbed on the left.  Thanks to the properties \eqref{propvarphi} of $\varphi$ we deduce
\beq
\label{vvv}
\|DV\|_{L^{m}(B_{r})}\le C\, \|f\|_{L^{m}(B_{R})}+\frac{C}{s-r}\|V\|_{L^{m}(B_{s}\setminus B_{r})}, \qquad R\le r<s\le 2R,
\eeq
for a constant $C=C(N, m)$.

Regarding the second assertion, we consider the linear operator $T(f, G)=DV$,  where $V$ solves
\[
\begin{cases}
\Div V=f,\\
\curl V=\sqrt{2}\, G,
\end{cases} 
\qquad f\in C^{\infty}_{c}(\R^{N}), \quad G\in C^{\infty}_{c}(\R^{N}; \R^{N}\land \R^{N}),
\]
which, as already noted, is represented in terms of Riesz transforms as
\[
T(f, G)_{kh}=-R_{h}\, R_{k}\, f-\sqrt{2} \, \sum_{j=1}^{N}R_{h}\, R_{j} \, G_{kj}.
\]
Estimate \eqref{lp} implies that $T$ has an extension $T:X_{m}\to Y_{m}$, where 
\[
X_{m}:=L^{m}(\R^{N})\times L^{m}(\R^{N}; \R^{N}\land \R^{N}),\qquad Y_{m}:=L^{m}(\R^{N}; \R^N\otimes\R^N),
\]
with the norms
 \[
\|(f, G)\|_{m}=\left(\int |f|^m+|G|_2^m\, dx\right)^{1/m},\qquad \|M\|_{m}=\left(\int |M|_{2}^{m}\, dx\right)^{1/m}
\] 
on $X_{m}$ and $Y_{m}$, respectively. For the complex interpolation spaces   it holds
\[
[X_{m_{1}}, X_{m_{2}}]_{\theta}=X_{m},\qquad \frac{1}{m}=\frac{1-\theta}{m_{1}}+\frac{\theta}{m_{2}}, \qquad \theta\in [0, 1],
\]
with equality of norms, and the same holds for the $Y_{m}$'s . On the other hand,  Lemma \ref{divrotl2} ensures that, with respect to these norms, 
\[
\|T\|_{{\cal L}(X_{2}, Y_{2})}= 1.
\]
Fix $\bar m'<2<\bar m$. The Riesz-Thorin interpolation theorem \cite[Theorem 1.19]{Duo} yields
\[
\|T\|_{{\cal L}(X_{m}, Y_{m})}\le \|T\|_{{\cal L}(X_{\bar m}, Y_{\bar m})}^{\theta},\qquad \frac{1}{m}=\frac{1-\theta}{2}+\frac{\theta}{\bar m}\, ,
\]
for any $2\le m\le \bar m$, and a similar estimate holds also for $\bar m'\le m\le 2$. 
We infer that there exist $\eta:[\bar m', \bar m]\to [0, +\infty)$ such that 
\beq
\label{eta123}
\lim_{m\to 2}\eta(m)=0
\eeq
and for all $m\in [\bar m', \bar m]$ it holds
\[
\|T\|_{{\cal L}(X_{m}, Y_{m})}\le 1+\eta(m).
\]
To complete the choice of $\delta_{0}$ in the second case, we proceed as in the proof of  \eqref{v}, setting $W=\varphi\, V$ and using the previous operator norm estimate, to get, for $\eta=\eta(m)$,
\[
\begin{split}
\|\varphi\, DV\|_{m}&\le \|DW\|_{m}+\|V\otimes D\varphi\|_{m}\\
&\le \|T(\Div W, 2^{-1/2} \curl W)\|_m+\|V\otimes D\varphi\|_{m}\\
&\le (1+\eta)\, \|(\Div W, 2^{-1/2}\curl W)\|_{m}+\|V\otimes D\varphi\|_{m}\\
&\le (1+\eta) \, \Big(\|(\varphi\, f, \varphi\, 2^{-1/2}\curl V)\|_{m}+\|\big((V, D\varphi), V\land D\varphi\big)\|_{m}\Big)+\|V\otimes D\varphi\|_{m}\\
&\le \frac{1+\eta}{\sqrt{2}}\, \|\varphi\, \curl V\|_{m}+ C\, \|\varphi\, f\|_{m}+\frac{C}{s-r}\, \|V\|_{L^{m}(B_{s}\setminus B_{r})}\\
&\hspace{-5pt}\underset{\eqref{vv}}{\le}(1+\eta)\, \e(K)\,  \|\varphi\, DV\|_{m}+C\, \|f\|_{L^{m}(B_{R})}+\frac{C}{s-r}\, \|V\|_{L^{m}(B_{s}\setminus B_{r})}.
\end{split}
\]
Since $\e(K)<1$, thanks to \eqref{eta123} we can choose $\delta_{0}=\delta_{0}(K, N)$ in such a way that 
\[
(1+\eta(m))\, \e(K)<1,\qquad \forall m\in [2-\delta_{0}, 2+\delta_{0}], 
\]
which again gives estimate \eqref{vvv} with a constant $C=C(N, K)$. 

Proceeding   as in the proof of Theorem \ref{Mth}, we see that estimate \eqref{vvv} improves to 
\beq
\label{finalest}
\|V\|_{W^{1, m}(B_{R})}\le C\, \|f\|_{L^{m}(B_{2R})}+C_{R, \theta}\|V\|_{L^{\theta}(B_{2R})}, \qquad \theta\in (0, m], 
\eeq
which therefore holds uniformly for all $V_{n}=DF_{n}(Dv_{n})$ constructed at the beginning. If $p, q$ are given in Proposition \ref{proqu}, point (ii), we set $\bar\theta=\min\{p/(q-1), m\}$  and proceed as in  \eqref{vvvv} to get a uniform bound on $\|V_n\|_{L^{\bar\theta}(B_{2R})}$. Thanks to \eqref{finalest}, the latter in turn implies the compactness of the $V_{n}$ in $L^{m}(B_{R})$. The rest of the proof of Theorem \ref{minloc} follows {\em verbatim}, providing estimates \eqref{v0bis} and \eqref{v0} in both the stated cases. We omit the details.

\begin{remark}
Closely inspecting the previous proof yields the following asymptotic estimates. The constant $ K_0$ goes to $1$ as $ m \to \infty $ or $m\to 1$. 
Similarly, \eqref{v0bis}, \eqref{v0} hold true for $m\in (a(K), b(K))$ with $a(K)\to 1$ and $b(K)\to +\infty$ as $K\to 1$.

We made no attempt to obtain optimal estimates for $K_0$ and $\delta_0$  as $ m \to 2 $, mainly due to the roughness of  estimate \eqref{lp} outlined in Remark \ref{nonoptimal}. Thus, we do not recover Theorem \ref{minloc} by simply letting $m\to 2$ in the previous statement.
\end{remark}

\end{proof}

\subsection{On the $C^{p'}$ conjecture}\label{sola}

An immediate corollary of the Cordes condition proved in the previous section is the following one.

\begin{corollary}\label{corcor}
Any weak solution $u\in W^{1, p}_{{\rm loc}}(\Omega)$ of $\Delta_{p}u=f\in L^{\infty}_{{\rm loc}}(\Omega)$ belongs to $C^{2 -\alpha}(\Omega)$, where $\alpha= \alpha(N, p)\le C(N)\, |p-2|$, provided $ |p-2| < 1/(2N^3) $.
\end{corollary}

\begin{proof}
Recall that $|z|^p$ is $K$-q.\,u.\,c.\,with constant
\[
K_p=\max\{p-1, 1/(p-1)\}=\begin{cases}
p-1& \text{if $p\ge 2$}\\
1/(p-1)&\text{if $p\in (1, 2)$}.
\end{cases}
\]
We apply the Cordes  estimates of the previous theorem, so that $|Du|^{p-2}\, Du\in W^{1, m}_{{\rm loc}}(\Omega)$ for any  $m\ge 2$ such that  (see  \eqref{condkzero}) 
\beq
\label{n3}
\sqrt{2} \, N^{2}\, (m-1)\, \big(1-1/K_p\big)<1.
\eeq
Given $p$, we let
\[
m_p=\frac{1}{2N^2\, |p-2|}.
\]
It is readily checked that,  for all $p$ such that $|p-2|<1/(2N^3)$, inequality \eqref{n3} holds for $m_p$ and moreover $m_p>N$. By the Morrey embedding  we thus have that $|Du|^{p-2}\, Du\in C^{1-N/m_p}(B)$.

The map $\Psi:\R^{N}\to \R^{N}$ defined as 
\beq
\label{defPsi}
\Psi(y)=
\begin{cases} |y|^{\frac{2-p}{p-1}}\, y &\text{if $y\ne 0$}\\
0 &\text{if $y=0$}
\end{cases}
\eeq
is the inverse of the similarly defined map $z\mapsto z^{p-1}=|z|^{p-2}z$, for which the well-known inequalities
\[
(z_{1}^{p-1}-z_{2}^{p-1}, z_{1}-z_{2})\ge 
\begin{cases} 
2^{2-p}\, |z_{1}-z_{2}|^{p}&\text{if $p\ge 2$}\\[3pt]
(p-1)|z_{1}-z_{2}|^{2}(1+|z_{1}|^{2}+|z_{2}|^{2})^{\frac{p-2}{2}}&\text{if $1<p< 2$}
\end{cases}
\]
hold true. By the Schwartz inequality we deduce
\[
|z_{1}^{p-1}-z_{2}^{p-1}|\ge 
\begin{cases}
2^{2-p}\, |z_{1}-z_{2}|^{p-1} &\text{if $p\ge 2$}\\
c_{M}\, |z_{1}-z_{2}|&\text{if $1<p< 2$ and $|z_{1}|+|z_{2}|\le M$},
\end{cases}
\]
which means that $\Psi$ is globally $1/(p-1)$-H\"older-continuous if $p\ge 2$ and locally Lipschitz-continuous if $1<p< 2$.
In the second case, we observe that $f\in L^{\infty}_{{\rm loc}}(\Omega)$ implies that $Du\in L^{\infty}_{{\rm loc}}(\Omega)$, hence $\Psi$ is Lipschitz-continuous on the range of $Du$. In both cases we thus have
\[
Du\in C^{\alpha_{p}}(\Omega), \qquad \alpha_{p}=
\begin{cases}
\displaystyle{\frac{1}{p-1}\Big(1-\frac{N}{m_{p}}\Big)=\frac{1-2 \, N^3\, (p-2)}{p-1}}&\text{if $p\ge 2$}\\[5pt]
1- 2 \, N^{3}\, (2-p)&\text{if $1<p<2$},
\end{cases}
\]giving the claim. 
\end{proof}

\begin{remark}
A similar conclusion can be drawn for $W^{1,p}_{\rm loc}(\Omega)$ local minimizers of $J(\cdot, \Omega)$ when $F$ is $K$-q.\,u.\,c.\,and $f\in L^\infty_{\rm loc}(\Omega)$. Indeed, $DF$ is $K^{N-1}$ quasiconformal, hence so is $DF^{-1}$. In particular, $DF^{-1}$ is $1/K$-H\"older continuous and the $\alpha$-H\"older regularity of $DF(Du)$ translates to $\alpha/K$-H\"older regularity for $Du$. The dependence of the H\"older exponent of $Du$ from $K$ turns out to be $1-c_N\,(K-1)$ for $K$ sufficiently near $1$. 
\end{remark}

Consider now a solution $u$ of the inhomogeneous elliptic equation with Uhlenbeck structure
\beq
\label{da}
\Div (a(|Du|)\, Du)=f \in L^{m}_{\rm loc}(\Omega), \quad m>1,
\eeq
where $a:(0, +\infty)\to (0, +\infty)$ is  $C^{1}(0, +\infty)$ and fulfills the ellipticity condition $-1<i_{a}\le s_{a}<+\infty$. In this framework, the exponent $p$ involved in applying the theorems of the previous sections can be found through Example \ref{exuhl} and Proposition \ref{proqu} and turns out to be 
\beq
\label{iasa}
p=1+\frac{1}{K}=\min\Big\{2+i_a, \frac{2+s_a}{1+s_a}\Big\}.
\eeq

Solutions of \eqref{da} are to be meant in a generalized sense and, as in \cite{CM}, we will use the notion of {\em approximable solutions} (or SOLA, solution obtained as limit of approximation) for \eqref{da}: $u$ is an approximable solution for \eqref{da} in $\Omega$ if  $a(|Du|)\, Du\in L^{1}_{{\rm loc}}(\Omega)$, \eqref{da} holds in the distributional sense in $\Omega$, and there exists a sequence $(f_{n})\subseteq C^{\infty}_{c}(\Omega)$ and corresponding weak solutions $u_{n}$ of \eqref{da} in $\Omega$ with right-hand side $f_{n}$ such that
\[
f_{n}\to f \quad \text{in $L^{m}_{{\rm loc}}(\Omega)$}, \qquad u_{n}\to u\ \text{ and }\   Du_{n}\to Du\quad \text{a.\,e.\,in $\Omega$},
\]
and
\[
\lim_{n}\int_{\Omega'}a(|Du_{n}|)\, |Du_{n}|\, dx=\int_{\Omega'}a(|Du|)\, |Du|\, dx
\]
for all $\Omega'\Subset\Omega$. Notice that $u$ may fail to belong to $W^{1, 1}_{{\rm loc}}(\Omega)$, but rather falls into the larger space 
\[
{\cal T}^{1, 1}_{{\rm loc}}(\Omega)=\Big\{v: T_{k}v \in W^{1, 1}_{{\rm loc}}(\Omega) \text{ for all $k>0$}\Big\},\qquad T_{k}v=\max\{-k, \min\{v, k\}\},
\]
for which a pointwise notion of $Du$ is well-defined almost everywhere. Whenever $f\in W^{-1,p' }(\Omega)$ (with $p$ given in \eqref{iasa}), weak and SOLA solutions coincide, and in particular $u$ belongs to the relevant Orlicz-Sobolev space. For the existence and uniqueness theory of SOLA we refer to \cite{CM3}.

\medskip
We say that $u\in {\cal T}^{1,1}_{{\rm loc}}(\Omega)$ is a cylindrical solution of \eqref{da} if there exists a point $x_{0}\in \R^{N}$ and a $k$-dimensional vector subspace $V\subseteq \R^{N}$ with corresponding orthogonal projection $\pi_{V}:\R^{N}\to V$, such that 
\[
u(x)=v(|\pi_{V}(x-x_{0})|)
\]
for some $v:I\to \R$ with $I\subseteq [0, +\infty)$ open and $\Omega\subseteq \{x\in V: |x|\in I\}\times V^{\perp}$. In other terms, a cylindrical function {\em only} depends on the distance from some vector subspace.

In order to study the regularity properties of a cylindrical solution of \eqref{da}, we first perform some straightforward reductions. It is clear that we can assume 
\[
V=\{x\in \R^{N}: x_{i}\equiv 0\  \forall i=k+1,\dots, N\}, \qquad \Omega=A\times \R^{N-k}
\]
with $A\subseteq \R^{k}$ invariant by the action of the orthogonal group $O_{k}$ on $\R^{k}$. Actually, by the structure of equation \eqref{da}, we can directly suppose that $k=N$, reducing to the case of radial solutions on a radial (meaning, invariant by $O_{N}$) domain $\Omega$. 

\begin{theorem}\label{CZA}
Let $a\in C^{1}((0, +\infty); (0, +\infty))$ fulfill \eqref{indexes} and $f\in L^{m}(\Omega)$ for some $m>1$, where $\Omega$ is a radial domain. If $u$ is a radial approximable solution of 
\[
\Div (a(|Du|)\, Du)=f
\]
in $\Omega$ then, for any $m>1$ and  $B_{R}$ such that $B_{2R}\Subset \Omega$, it holds
\beq
\label{cze}
\|a(|Du|)\, Du\|_{W^{1, m}(B_{R/2})}\le C_{m, R}\, \big(\|f\|_{L^{m}(B_{2R})}+\|a(|Du|)\, Du\|_{L^{1}(B_{2R})}\big).
\eeq
\end{theorem}

\begin{proof}
The field $V=a(|Du|)\, Du\in L^{1}_{{\rm loc}}(\Omega)$ is the pointwise limit of the fields $V_{k}=a(|DT_{k}u|)\, DT_{k} u$ for $k\to +\infty$, which satisfy
\beq
\label{radial}
 T\circ V_k=V_k\circ T \quad \forall T\in O_{N},\qquad |(V_k(x), x)|=|V(x)|\, |x|,
 \eeq
 hence it satisfies the latters too.
  We extend $V$ and $f$ as zero outside $\Omega$ (thus keeping the previous properties for $V$)  and let $V_{\eps}=V*\varphi_{\eps}$, $f_{\eps}=f*\varphi_{\eps}$,  where $\varphi_{\eps}$ is a standard radial convolution kernel supported in $B_{\eps}$. We claim that $V_\eps$ obeys \eqref{radial}. By changing variables and using the radiality of $\varphi_{\eps}$, it is readily checked that $V_{\eps}$ satisfies $T\circ V_\eps=V_\eps\circ T$ for all $T\in O_N$. Thus, in order to check the second condition in \eqref{radial} it suffices to prove it at a point $x=r\, e_1$, where
  \[
V_\eps(r\, e_1)=\int V(y) \, \varphi_\eps(|r\, e_1-y|)\, dy.
\]
The integrand above is odd with respect to the reflections $y_k\mapsto -y_k$,  $k=2, \dots, N$, which implies that $V_\eps(r\, e_1)$ is parallel to $e_1$, and this concludes the proof of \eqref{radial} for $V_\eps$. It follows that for 
\[
h_\eps(x):=(V_\eps(x), x/|x|^2)\in C^\infty(\R^N)
\]
it holds, with a slight abuse of notation, $h_\eps(x)=h_\eps(|x|)$ and
\beq
\label{vepsr}
V_{\eps}(x)=h_{\eps}(|x|)\, x.
\eeq
 Given a radial subdomain $\Omega' \Subset \Omega$ and using Fubini's theorem, we have
\[
\begin{split}
\int_{\Omega'} (V_{\eps}, D\psi)\, dx &=\int_{\Omega'} (V,  \varphi_{\eps}*D\psi)\, dx =\int_{\Omega'} (V, D(\psi*\varphi_{\eps}))\, dx\\
&=-\int_{\Omega'}f\, \psi*\varphi_{\eps}\, dx=-\int_{\Omega'} f_{\eps}\, \psi \, dx,
\end{split}
\]
thus $V_{\eps}$ fulfills $\Div V_{\eps}=f_{\eps}$
weakly (and thus strongly) in $\Omega'$ for all sufficiently small $\eps>0$. From \eqref{vepsr}  we compute, for $x\ne 0$,
\[
DV_{\eps}(x)=h_{\eps}(|x|)\,  I+|x|\, h_{\eps}'(|x|)\, \frac{x}{|x|}\otimes \frac{x}{|x|},
\]
which is a symmetric matrix, so that 
\[
\curl V_{\eps}=0 \qquad \text{in $\Omega'$}.
\]
By the Poincaré lemma, for any $B_{2R}\subseteq \Omega'$ and sufficiently small $\eps>0$, we thus have $V_{\eps}=D v_{\eps}$  for some $v_{\eps}\in C^{2}(B_{2R})$, fulfilling weakly $\Delta v_{\eps}=f_{\eps}$. For $R<r<s<2\, R$ we choose  cut-off functions as in \eqref{propvarphi} and suppose, without loss of generality, that $v_{\eps}$ has zero mean in $B_{s}$. By the standard Calder\'on-Zygmund estimates (see  \cite[Theorem 5.1]{Duo}) and Poincar\'e's inequality, it holds
\[
\begin{split}
\|D^{2}v_{\eps}\|_{L^{m }(B_{r})}^{m}&\le C_{m}\, \Big(\|f_{\eps}\|_{L^{m}(B_{2R})}^{m}+\frac{1}{(s-r)^{m}}\|Dv_{\eps}\|_{L^{m}(B_{s})}^{m}+\frac{1}{(s-r)^{2\, m}}\|v_{\eps}\|_{L^{m}(B_{s})}^{m}\Big)\\
&\le C_{m}\, \|f_{\eps}\|_{L^{m}(B_{2R})}^{m}+C_{m}\Big(\frac{1}{(s-r)^{m}}+\frac{ R^{m}}{(s-r)^{2\, m}}\Big)\, \|Dv_{\eps}\|_{L^{m}(B_{s})}^{m}.
\end{split}
\]
Hence,  we can proceed as in the final part of the proof of Theorem \ref{Mth} to improve the latter to 
\[
\|V_{\eps}\|_{W^{1,m}(B_{R/2})} \le C_{m}\, \|f_{\eps}\|_{L^{m}(B_{2R})}+ C_{m, R}\, \|V_{\eps}\|_{L^{1}(B_{2 R})}.
\]
Since $V_{\eps}\to V$ in $L^{1}(B_{R})$, we obtain the claimed estimate \eqref{cze} by lower semicontinuity. 
\end{proof}

\begin{corollary}\label{radialcp'}
Let $u\in W^{1, p}(B_R)$ be a cylindrical weak solution of 
\[
\Delta_{p}u=f \ \in L^{\infty}(B_{R}).
\]
Then 
\begin{enumerate}
\item If $p\ge 2$, $u\in C^{p'-\eps}_{{\rm loc}}(B_{R})$ for all $\eps>0$, and also for $\eps=0$ if furthermore $f\in C^{0}_{{\rm Dini}}(B_{R})$.
\item
 If $p\le 2$, $u\in C^{2-\eps}_{{\rm loc}}(B_{R})$ for all $\eps>0$, and also for $\eps=0$ if furthermore $f\in C^{0}_{{\rm Dini}}(B_{R})$.
 \end{enumerate}
\end{corollary}

\begin{proof}
According to Theorem \ref{CZA}, the field $V=|Du|^{p-2}\, Du$ belongs to $W^{1, m}_{{\rm loc}}(B_{R})$ for any $m>1$; hence, by the Morrey embedding, it lies in $C^{1-\eps}(B_{R})$ for any $\eps>0$. Similarly, $V\in {\rm Lip}_{{\rm loc}}\, (B_{R})$ if $f$ is Dini-continuous, being the gradient of a solution of $\Delta v=f\in C^{0}_{{\rm Dini}}(B_{R})$. We conclude through the properties of the map $\Psi$ in \eqref{defPsi}, as in the proof of Corollary \ref{corcor}.
\end{proof}

\medskip
\noindent
{\small {\bf Acknowledgements.} We warmly thank Lorenzo Brasco, Giulio Ciraolo and Alberto Farina for fruitful conversations on some of the themes of this research.  We also thank an anonymous referee for several suggestions which improved the quality of the manuscript.

S.\,Mosconi is supported by grant PRIN n.\ 2017AYM8XW of the Ministero dell'Istruzione, Universit\`a e Ricerca and grant PIACERI 20-22 - linea 2 and linea 3 of the University of Catania.}

\end{document}